\documentclass[11pt]{amsart} 

\usepackage{verbatim, latexsym, amssymb, amsmath,color}
\usepackage{epsfig}
\usepackage{CJK}
\usepackage{amsmath}
\usepackage{amssymb}
\usepackage{amsthm}
\usepackage[utf8]{inputenc}
\usepackage[english]{babel}

\newtheorem{theorem}{Theorem}[section]
\newtheorem{corollary}[theorem]{Corollary}
\newtheorem{proposition}[theorem]{Proposition}

\newtheorem{lemma}[theorem]{Lemma}
\newtheorem{claim}[theorem]{Claim}
\newtheorem{definition}[theorem]{Definition}
\newtheorem{remark}[theorem]{Remark}
\newcommand{\at}[2][]{#1|_{#2}}
\usepackage{geometry}
\title{Morse Index Bound for Minimal Two Spheres}
\author{Yuchin Sun}

\begin{document}
	\maketitle
	
	\begin{abstract}
Given a closed manifold of dimension at least three, with non trivial homotopy group $\pi_3(M)$ and a generic metric, we prove that there is a finite collection of harmonic spheres with Morse index bound one, with sum of their energies realizes a geometric invariant width.
	\end{abstract}
	\section{Introduction}
	Finite-dimensional Morse theory was developed by Morse \cite{Morse} to study geodesics. Index of a critical point of a proper nonnegative function on a manifold reflects its topology. A natural extension of Morse theory of closed geodesics would be a Morse theory of harmonic surfaces in a Riemannian manifold. Sacks and Uhlenbeck introduced $\alpha$-energy \cite{SU}, which can be perturbed to be Morse functions. The $\alpha$-energy approaches the usual energy as the parameter $\alpha$ in the perturbation goes to one, and the corresponding critical point of $\alpha$-energy converges to a harmonic map. However, without curvature assumption \cite{LLW} or finite fundamental group \cite{CT} for the ambient manifold $(M,g)$, the harmonic sphere constructed by $\alpha$-energy fails to realize the energy as $\alpha$ goes to one \cite{YY}\cite[Remark 4.9.6]{JD}. Thus, we are motivated to prove the Morse index bound of the harmonic sphere produced by the min-max theory \cite{CD}, which rules out the energy loss, namely:
\begin{theorem}[Main Theorem]\label{main}
Let $(M,g)$ be a closed Riemannian manifold with dimension at least three, $g$ generic and a nontrivial homotopy group $\pi_3(M).$ Then there exists a collection of finitely many harmonic spheres $\{u_i\}_{i=0}^n,\: u_i:S^2\to M$, which satisfies the following properties:
\begin{enumerate}
\item $\sum_{i=0}^n E(u_i)=W,$
\item $\sum_{i=0}^n Index(u_i)\leq 1,$
\end{enumerate}
here $W$ is a geometric invariant called \textit{width}.(See Definition \ref{width}.) 
\end{theorem}
The collection of finitely many harmonic spheres in Theorem \ref{main} is constructed by Colding and Minicozzi \cite{CD} by the min-max theory for the energy functional, which they used to prove finite extinction time of Ricci flow. The theory can be loosely spoken as the following: given a closed manifold $M$, sweep $M$ out by a continuous one parameter family of maps from $S^2$ to $M$, starting and ending at point maps. Pull the \textit{sweepout} (Definition \ref{width}) tight, in a continuous way, by harmonic replacements. Then if the sweepout induces a nontrivial class in $\pi_3(M)$, then each map in the tightened sweepout whose area is close to the width must itself be \textit{close} to a finite collection of harmonic spheres, close in the bubble tree sense (Definition \ref{bubble conv}). In other words, the width is realized by the sum of areas of finitely many harmonic spheres. Theorem \ref{main} states that the sum of Morse indices of the harmonic spheres is at most one.

	On the other hand, harmonic spheres are minimal surfaces.  Almgren and Pitts' min-max theory \cite{AP} proves the existence of embedded minimal hypersurfaces in closed manifold of dimension at least three at most seven. Marques and Neves proved the Morse index bound of such an embedded minimal hypersurface \cite{MN}. This result plays an important role in proving Yau's conjecture \cite{AS}, which states that for any closed three manifold there exist infinitely many embedded minimal surfaces. However, Almgren and Pitts' min-max theory doesn't say anything about minimal surfaces in higher codimension. While using the min-max theory of harmonic spheres by Colding and Minicozzi \cite{CD}, there's no restriction on codimension of the ambient manifold. So it motivates us to prove the Morse index bound of the harmonic sphere produced by the min-max theory of Colding and Minicozzi \cite{CD}.
	
	Theorem \ref{main} seems like a variant of \cite{MN}. We compare the difference between them here. Besides the obvious difference of harmonic spheres and embedded minimal hypersurfaces, codimension restriction of ambient manifold, the embedded minimal hypersurface used in \cite{MN} is given by Almgren-Pitts' min-max theory, thus could have several components. When considering the variation of it, it means the variation of the whole configuration instead of each component. The index of \cite{MN} is the maximal dimension on which the second variation of the area functional of the whole configuration is negative definite. (But since the components are disconnected embedded minimal hypersurfaces, the index of the whole configuration is equivalent to the sum of Morse indices of each component.) While in theorem \ref{main}, the finite collection of harmonic spheres are not necessarily disconnected, the index of the whole configuration is less than or equal to the Morse index sum of each harmonic sphere. But the index bound we obtain in theorem \ref{main} is the sum of Morse indices of each component, which is stronger than the bound for the whole configuration. 

We also mention the following Morse Index conjecture proposed by Marques and Neves \cite{MNeves}:
\\
\textbf{Morse Index Conjection}
\textit{For generic metric on $M^{n+1}$, $3\leq (n+1)\leq 7$, there exists a smooth, embedded, two-sided and closed minimal hypersurface $\Sigma$ such that $Index(\Sigma)=k$ for any integer $k$.}
\\
Marques and Neves have shown that the conjecture is true under the assumption of multiplicity one \cite{MNeves}. For harmonic spheres, we consider the case of same assumption in theorem \ref{main}, the conjecture is true if all the finite collection of harmonic spheres in the image set (see definition \ref{image set}) is one harmonic sphere. In other words, if the min-max sequence (see definition \ref{minimizing sequence}) converges to only one harmonic sphere strongly, then that harmonic sphere has Morse index one. For the general case, the difficulty lies in bubble convergence (definition \ref{bubble convergence}). Ideally, we want to use the idea that a local minimizer can't be a min-max limit, and a stable harmonic sphere is a local minimizer for energy functional among all the spheres lie in the small tubular neighborhood. But bubble convergence doesn't imply the min-max sequence lies in the tubular neighborhood of one harmonic sphere, thus making it hard to conclude that the harmonic spheres in the image set can't all be stable.

\subsection{Idea of the Proof for Theorem \ref{main}}

We consider the image set $\Lambda(\{\gamma_j(\cdot,t)\}_j)$ of a minimizing sequence $\{\gamma_j(\cdot,t)\}_{j\in\mathbb{N}}$. The idea is if $\{u_i\}_{i=0}^m\in\Lambda(\{\gamma_j(\cdot.t)\}_j)$ have $\sum_{i=0}^m\text{Index}(u_i)>1$, then we are able to perturb $\{\gamma_j(\cdot,t)\}_{j}$ to a new sweepout $\{\tilde{\gamma}_j(\cdot,t)\}_j$ such that it is homotopic to $\gamma_j$, it is a minimizing sequence, and $\{u_i\}_{i=0}^m$ is not in its image set. Since $\{\tilde{\gamma}_j(\cdot,t)\}_j$ is minimizing, $\Lambda\{\tilde{\gamma}_j(\cdot,t)\}_j$ is nonempty. If for $\{v_i\}_{i=0}^{m_0}\in\Lambda(\{\tilde{\gamma}_j(\cdot,t)\}_j)$ we have $\sum_{i=0}^{m_0}\text{Index}(v_i)>1$, then we can perturb $\{\tilde{\gamma}_j(\cdot,t)\}_j$ again and get a new sweepout such that neither $\{u_i\}_{i=0}^m$ nor $\{v_i\}_{i=0}^{m_0}$ is in its image set. Proposition \ref{countable} states that the set of harmonic spheres with bounded energy $W$ is countable, which allows us to perturb the sweepout inductively and get a sweepout which is away from all harmonic spheres whose sum of Morse indices is greater than $1$. Since it is a minimizing sequence, it converges to a collection of finitely many harmonic spheres whose sum of Morse indices is bounded by one.

Before constructing $\{\tilde{\gamma}_j(\cdot,t)\}_j$, we define the variation of a map. Suppose $M$ is isometrically embedded in $\mathbb{R}^N$ and let $\Pi:\mathbb{R}^N\to M$ be the nearest point projection from $\mathbb{R}^N$ to $M$. Given a map $u:S^2\to M$ and $X:S^2\to \mathbb{R}^N$, with each $X^i\in C^{\infty}(S^2)$, we consider the variation of $u$ with respect to $X$ to be 
$u_s:=\Pi\circ(u+sX)$. We choose to define the variation this way so that for any map $v:S^2\to M$ close to $u$ in $W^{1,2}(S^2,M)$, the variation $v_s$ is close to $u_s$ as well. 

 Assume $\{u_i\}_{i=0}^m\in\Lambda(\{\gamma_j(\cdot,t)\}_j)$ with $\sum_{i=0}^m\text{Index}(u_i)=k\geq 2$, then there exists $\{X_l\}_{l=1}^k$, $X_l:S^2\to\mathbb{R}^N$, with the following property: for each $X_l$, there exists at least one $u_l\in\{u_i\}_{i=0}^m$ so the second variation of energy of $u_l$ with respect to $X_l$ is negative. The idea is using $\{X_l\}_{l=1}^k$ to perturb $\gamma_j(\cdot,t)$. We first prove in lemma \ref{unstable} that for $\gamma_j(\cdot,t)$ close to $\{u_i\}_{i=0}^m$, there exist corresponding cutoff functions $\eta^j_l:S^2\to\mathbb{R}$. Let $\tilde{X}_l:=\eta_l^jX_l$ and define the variation of $\gamma_j(\cdot,t)$ with respect to $\tilde{X}_l$ to be:
\[\gamma_{j,s}(\cdot,t):=\Pi\circ\big(\gamma_j(\cdot,t)+\sum_{l=1}^ks_l\tilde{X}_l\big),\]
here $s=(s_1,...,s_k)\in\bar{B}^k$, $\bar{B}^k$ is the $k$-dimensional unit ball, so that the energy of $\gamma_{j,s}(\cdot,t)$ is concave while changing $s\in\bar{B}^k$. That is, define $E_j^t(s):=E(\gamma_{j,s}(\cdot,t)), E_j^t:\bar{B}^k\to\mathbb{R},$ and we have
\[D^2E_j^t(s)<0,\quad\forall s\in\bar{B}^k.\]
If we can construct a continuous function $s_j:[0,1]\to\bar{B}^k$ so that energy decreases by a certain amount when 
$\gamma_j(\cdot,t)$ is close to $\{u_i\}_{i=0}^m$. Then the sequence $\{\gamma_{j,s_j(t)}(\cdot,t_j)\}_j$ does not converge to $\{u_i\}_{i=0}^m$, and $\gamma_{j,s_j(t)}(\cdot,t)$ is the desired sweepout. In order to construct $s_j(t)$, we observe the following one parameter gradient flow $\{\phi_j^t(\cdot,x)\}\in\text{Diff}(\bar{B}^k)$, with $x\in\bar{B}^k$ as starting point, generated by the vector field:
\[s\mapsto -(1-|s|^2)\nabla E_j^t(s),\quad s\in\bar{B}^k.\]
$\phi_j^t(\cdot,x)$ decreases the energy, except when $|x|=1$ or $x$ is the maximal point of $E_j^t$. The assumption of the lower bound of Morse index $\sum_{i=0}^m\text{Index}(u_i)=k>1$, now enables us to construct a continuous curve $y_j:[0,1]\to\bar{B}^k$ avoiding the maximal point of the function $E^t_j$ as $t$ varies. Namely, let $\nabla E_j^t(x(t))=0$, $x_j:[0,1]\to\bar{B}^k$, $x(t)$ is a continuous curve on $\bar{B}^k$. Since the dimension of $\bar{B}^k$ is larger than $1$, we can choose a continuous curve $y_j(t)$ on $\bar{B}^k$ which does not intersect with $x_j([0,1])$, then we can use $\{\phi_j^t(\cdot,y_j(t))\}$ to construct $s_j(t)$ and obtain the new sweepout $\tilde{\gamma}_j(\cdot,t):=\gamma_{j,s_j(t)}(\cdot,t)$. The new sweepout is homotopic to $\{\gamma_j(\cdot,t)\}_j$, and doesn't bubble converge to $\{u_i\}_{i=0}^m$. This is the desired perturbed sweepout.

The organization of the paper is as follows. In section 2 we give the basic definitions of harmonic sphere, bubble convergence, and state the min-max theorem \ref{CDDD}. In section 3 we prove a technical lemma \ref{unstable}. In section 4 we prove the main result theorem \ref{main}.

\subsection*{Acknowledgments}
	I would like to express my gratitude towards my advisor Andr\'e Neves. This work would not have been possible without his generous support and insightful guidance. I would like thank Xin Zhou for many valuable discussions and inspiring suggestions.
	
	\section{Background Material}	
\subsection{Harmonic Sphere}\label{WeaklyHarmonicMap}
Suppose that $S^2$ is a Riemann sphere, which can be regarded as $\mathbb{C}\cup\{\infty\}$, and $M$ is a closed manifold of dimension at least three, isometrically embedded in $\mathbb{R}^N$.

We introduce  \textit{nearest point projection} $\Pi:\mathbb{R}^N\to M$ which maps a point $x\in\mathbb{R}^N$ to the nearest point of $M$. There is a tubular neighborhood of $M$
\[M_\delta=\big\{x\in\mathbb{R}^N:\text{ dist}(x,M)<\delta\big\},\]
on which $\Pi$ is well-defined and smooth. For a map $u:S^2\to M\subseteq\mathbb{R}^N$, $u=(u^1,u^2,..,u^N)$, and $u\in W^{1,2}(S^2,M)$, we write $\nabla u$ as the sum of \textit{gradient} of $u^i$ for $i=1,...,N$. That is, $\nabla u:=\sum_{i=1}^N\nabla u^i$, and \textbf{energy} of $u$ is simply
\begin{equation}\label{energyeq}
 E(u)=\int_{S^2}|\nabla u|^2
\end{equation}
 For a given $X\in C^{\infty}(S^2,\mathbb{R}^N)$, we consider the variation of $u$ with respect to $X$ defined as the following:
\begin{equation}\label{whm}
	u_{s}=\Pi\circ(u+sX),
\end{equation}
 $u_s$ is well-defined for $s$ small enough such that the image of $u+sX$ is in the tubular neighborhood $M_\delta$.
\begin{definition}[Harmonic Sphere]\label{WHS}
We say that $u\in W^{1,2}(S^2,M)$ is a harmonic sphere if for any $X\in C^{\infty}(S^2,\mathbb{R}^N)$ we have
\begin{equation}\label{2.27}
\lim_{s\to 0}\frac{E(u_s)-E(u)}{s}=0.   
\end{equation}
\begin{remark}
Harmonic sphere is smooth  \cite{Helenf}.
\end{remark}
\end{definition}

Given a map $u:S^2\to M$, $u\in W^{1,2}(S^2,M)$, and $X\in C^{\infty}(S^2,\mathbb{R}^N)$, by Taylor polynomial expansion of $\Pi$ we have the following:
\begin{equation}\label{2.28}
\begin{split}
u_s&=\Pi\circ(u+sX)\\
&=u+sd\Pi_u(X)+\frac{s^2}{2}\text{Hess}\Pi_u(X,X)+o(s^2).     
\end{split}
\end{equation}
By applying $\nabla$ to (\ref{2.28}) we have
\begin{align*}
\nabla u_s=&\nabla u\\
	&+s(d\Pi_u(\nabla X)+\text{Hess}\Pi_u(X,\nabla u))\\
	&+\frac{s^2}{2}\big(2\text{Hess}\Pi_u(X,\nabla X)+\nabla \text{Hess}\Pi_u(X,X,\nabla u)\big)+o(s^2),
	\end{align*}
and the energy of $u_s$ is
\begin{equation}\label{2.229}
    \begin{split}
        E(u_s)=&\int_{S^2}\langle\nabla u,\nabla u\rangle \\
        +&s\int_{S^2}\langle\nabla u,d\Pi_u(\nabla X)\rangle+\langle\nabla u,
        \text{Hess}\Pi_u(X,\nabla u)\rangle \\
        +&s^2\Big\{\int_{S^2}\langle d\Pi_u(\nabla X),d\Pi_u(\nabla X)\rangle \\
	&+2\int_{S^2}\langle \text{Hess}\Pi_u(X,\nabla u),d\Pi_u(\nabla X)\rangle \\
	&+\int_{S^2}\langle \text{Hess}\Pi_u(X,\nabla u),\text{Hess}\Pi_u(X,\nabla u)\rangle \\
	&+\frac{1}{2}\int_{S^2}\langle\nabla u,\big( 2\text{Hess}\Pi_u(X,\nabla X)+\nabla\text{Hess}\Pi_u(X,X,\nabla u)\big)\rangle \Big\}\\
        +&o(s^3).
    \end{split}
\end{equation}
From (\ref{2.229}), we see that the first variation of energy is
\begin{equation}
    \begin{split}
        \delta E(u)(X):=&\frac{d}{ds}\at[\Big]{s=0}E(u_s)\\
        =&\int_{S^2}\langle\nabla u,d\Pi_u(\nabla X)\rangle+\langle\nabla u,
        \text{Hess}\Pi_u(X,\nabla u)\rangle \\
        =&\int_{S^2}\langle\nabla u,
        \nabla X\rangle-\langle X,A(\nabla u,\nabla u)\rangle.
    \end{split}
\end{equation}
the last equality follows from \cite[2.12.3]{LS}, and $u$ is a harmonic sphere if and only if 
\begin{equation}\label{whel}
        \Delta u+ A(\nabla u,\nabla u)=0.  
\end{equation}
The second variation of energy is:
\begin{equation}\label{2.29}
\begin{split}
	\delta^2 E(u)(X,X):=&\frac{d^2}{ds^2}\at[\Big]{s=0}E(u_s)\\
	=&\int_{S^2}\langle d\Pi_u(\nabla X),d\Pi_u(\nabla X)\rangle \\
	&+2\int_{S^2}\langle \text{Hess}\Pi_u(X,\nabla u),d\Pi_u(\nabla X)\rangle \\
	&+\int_{S^2}\langle \text{Hess}\Pi_u(X,\nabla u),\text{Hess}\Pi_u(X,\nabla u)\rangle \\
	&+\frac{1}{2}\int_{S^2}\langle\nabla u,\big( 2\text{Hess}\Pi_u(X,\nabla X)+\nabla\text{Hess}\Pi_u(X,X,\nabla u)\big)\rangle .
	\end{split}
	\end{equation}
It's clear that from ($\ref{2.29}$) we have 
\begin{equation}
|\delta^2E(u)(X,X)-\delta^2E(v)(X,X)|<\Psi(\|u-v\|_{W^{1,2}}),    
\end{equation}
for some continuous function $\Psi:[0,\infty)\to[0,\infty)$ with $\Psi(0)=0$.
If $u$ is a harmonic sphere, we have
\begin{equation}
    \begin{split}
   \delta^2E(u)(X,X)=&\int_{S^2}\langle\nabla d\Pi_u(X),\nabla d\Pi_u(X)\rangle\\
        &-\int_{S^2} \langle R^M(\nabla u,d\Pi_u(X))d\Pi_u(X),\nabla u\rangle.
    \end{split}
\end{equation}
\begin{definition}[Index Form]
The index form of a harmonic sphere $u:S^2\to M$ is defined by
\begin{equation}
    \begin{split}
    I(X,Y)=&\int_{S^2}\langle\nabla d\Pi_u(X),\nabla d\Pi_u(Y)\rangle\\
        &-\int_{S^2} \langle R^M(\nabla u,d\Pi_u(X))d\Pi_u(Y),\nabla u\rangle,
    \end{split}
\end{equation}
for $X,Y\in C^\infty(S^2,\mathbb{R}^N)$
\end{definition}
\begin{definition}[Index]\label{indexdef} The index of a harmonic sphere $u:S^2\to M$ is the maximal dimension of the subspace $X$ of $\Gamma(u^{-1}TM)$ on which the index form is negative definite.
\end{definition}
\begin{remark}[\cite{LS}]
For any $X\in C^\infty(S^2,\mathbb{R}^N)$, 
\[d\Pi_u(X)\in\Gamma(u^{-1}TM).\]
\end{remark}
\subsection{Bubble convergence of harmonic sphere}

This section is for defining bubble convergence (definition \ref{bubble convergence}, definition \ref{bubble norm}) and establishing several properties of it (proposition \ref{for unstable}). They are used for describing how close two maps are, which is essential for lemma \ref{unstable} and theorem \ref{deformation}. The varifold distance used by Colding and Minicozzi \cite{CD} is not sufficient because it only implies closeness in measure sense on the Grassmannian bundle of the ambient manifold. But if a map $\gamma$ is close to a finite collection of maps $\{u_i\}_{i=0}^n$ in bubble tree sense, that means for each $u_i$ there exist conformal dilation $D_i$ and compact domain $\Omega_i$ such that $\gamma$ is close to $u_i\circ D_i$ on $\Omega_i$ in $W^{1,2}$ sense. We state this \textit{closeness} of bubble convergence in definition \ref{bubble norm}, prove that it implies varifold convergence. Moreover, if a map $\gamma$ is close to $\{u_i\}_{i=0}^n$ in bubble tree sense if and only if it's close in varifold sense and the term  $\inf\Big\{\int_{S^2}|\nabla\gamma-\nabla\big(\sum_{i=0}^nu_i\circ\phi_i\big)|^2\Big|\phi_i\in PSL(2,\mathbb{C})\Big\}$ is small (see proposition \ref{for unstable}), this result is used in theorem \ref{deformation}.

\begin{definition}[M\"obius transformations]\label{ball}		
The group of automorphisms of the Riemann sphere is known as $PSL(2,\mathbb{C})$, it's also known as the group of \textit{M\"obius transformations}. Its elements are fractional linear transformations
\[\phi(z)=\frac{az+b}{cz+d},\quad ad-bc\neq 0,\]
where $a,b,c,d\in\mathbb{C}.$
	\end{definition}

	\begin{definition}[Bubble Convergence]\label{bubble convergence}
		We will say that a sequence $\gamma_{j}:S^2\to M$ of $W^{1,2}$ maps \textit{bubble converges} to a collection of $W^{1,2}$ maps $u_0,...,u_m:S^2\to M$ if the following hold:
		
		\begin{enumerate}
			\item The $\gamma_{j}$ converges weakly to $u_0$ and there's a finite set $\mathcal{S}_0=\{x_0^1,...,x_0^{k_0}\}\subset S^2$ so that the $\gamma_{j}$ converge strongly to $u_0$ in $W^{1,2}(K)$ for any compact set $K\subset S^2\setminus\mathcal{S}_0$.
			\item For each $i>0$, we get a point $x_{l_i}\in\mathcal{S}_0$ and a sequence of balls $B_{r_{i,j}}(y_{i,j})$ with $y_{i,j}\to x_{l_i}$. Further more, let $D_{i,j}$ be the dilation that takes the southern hemisphere to $B_{i,j}(y_{i,j})$. Then the map $\gamma_{j}\circ D_{i,j}$ converges to $u_i$ as in 1.
			\item if $i_1\neq i_2$, then $\frac{r_{i_1,j}}{r_{i_2,j}}+\frac{r_{i_2,j}}{r_{i_1,j}}+\frac{|y_{i_1,j}-y_{i_2,j}|^2}{r_{i_1,j}r_{i_2,j}}\to\infty.$
			\item $\sum_{i=0}^m E(u_i)=\lim\limits_{j\to\infty}E(\gamma_{j}).$
		\end{enumerate}
	\end{definition}
	We introduce $d_B(\cdot,\cdot)$ here to describe bubble convergence precisely. Notice that $d_B(\cdot,\cdot)$ is not a norm like $\|\cdot\|_{W^{1,2}}$ or $d_V(\cdot,\cdot)$ (see definition \ref{varifold distance}), we are abusing the notation here by using $d_B(\cdot,\cdot)$. 
	\begin{definition}\label{bubble norm}
		Given a collection of finitely many harmonic spheres $\{u_i\}_{i=0}^n$, and let $E=\sum_{i=0}^nE(u_i)$. For $\gamma:S^2\to M$, we say \[d_B(\gamma,\{u_i\}_{i=0}^n)<\epsilon,\]
		if we can find conformal dilations $D_{i}:S^2\to S^2,\:i=0,...,n,$ and pairwise disjoint domains $\Omega_{0},...,\Omega_{n}$, $\bigcup_{i=0}^n\Omega_{i}\subset S^2$ so the following holds:
		\begin{equation}\label{3.11}
		\sum_{i=0}^n\Big(\int_{\Omega_{i}}|\nabla(\gamma-\nabla (u_i\circ D_{i})|^2\Big)^{1/2}<\epsilon,
		\end{equation}
		\begin{equation}\label{3.12}
		\int_{S^2\setminus\bigcup_{i=0}^n\Omega_{i}}|\nabla\gamma|^2<3\epsilon^2+2(n+1)\epsilon E,
		\end{equation}
		\begin{equation}\label{3.13}
		\sum_{i=0}^n\Big(\int_{S^2\setminus\Omega_{i}}|\nabla(u_i\circ D_{i})|^2\Big)^{1/2}<\epsilon,
		\end{equation}

		We write $d_B(\gamma,\{u_i\}_{i=0}^n)\geq\epsilon$ if there's no pairwise disjoint domains $\{\Omega_i\}_{i=0}^n$ and conformal dilations $\{D_i\}_{i=0}^n$ satisfying (\ref{3.11}), (\ref{3.12}), and (\ref{3.13}).
	\end{definition}

	\begin{theorem}[Bubble convergence for harmonic maps, \cite{TP}]\label{bubble conv}
		Let	$u_i:\Sigma\to M$ be a sequence of harmonic maps from a Riemann surface to a compact Riemannian manifold with bounded energy $E_0$. i.e., 
		\[E(u_i)\leq E_0.\]
		Then $u_i$ bubble converges to a finte collection of harmonic maps $\{v_j\}_{j=0}^m$ Moreover, 
		\[\lim\limits_{i \to\infty}E(u_i)=\sum_{j=0}^mE(v_j).\]
	\end{theorem}
	Actually the sequence doesn't need to be harmonic. It also works for \textit{almost harmonic maps} like stated in Theorem \ref{almost harmonic}.

Now we introduce varifold distance and state the relation between bubble convergence and varifold convergence. The following definition of varifold distance $d_V(\cdot,\cdot)$ can be found at \cite[Chapter 3]{CDD}.
	\begin{definition}[Varifold Distance]\label{varifold distance}
		Fix a closed mainifold $M$, let \[P_\Pi:\mathcal{G}_kM\to M\] 
		be the Grassmannian bundle of (unoriented) k-planes, that is, each fiber $P_\Pi^{-1}(p)$ is the set of all k-dimensional linear subspaces of the tangent space of $M$ at $p$. Since $\mathcal{G}_kM$ is compact, we can choose a countable dense subset ${h_n}$ of all continuous functions on $\mathcal{G}_kM$ with supremum norm at most one. Given a finite collection of maps 
		\begin{align*}
		    f_i&:X_i\to M,\quad f_i\in W^{1,2}(X_i,M),i=1,...,k_1,\\
		    g_j&:Y_j\to M,\quad g_j\in W^{1,2}(Y_j,M),j=1,...,k_2,
		\end{align*}
		here $\{X_i\}_{i=1}^{k_1},\{Y_j\}_{j=1}^{k_2}$ are compact surfaces of dimension $k$. 
		We consider the pairs $\{(X_i,F_i)\}_{i=1}^{k_1}$ and $\{(Y_j,G_j)\}_{j=1}^{k_2}$ with measurable maps \[F_i:X_i\to\mathcal{G}_kM,\quad\text{and   }G_j:Y_i\to\mathcal{G}_kM,\]
		so that 
		\[f_i=P_\Pi\circ F_i,\quad\text{and    }g_j=P_\Pi\circ G_j.\]
		 ($F_i(x)$ is the linear subspace of $df_i(T_xM)$.) $J_{f_i}$ is the Jacobian of $f_i$, then the varifold distance between them is defined by:
		\begin{equation}\label{2.49}
		d_V(\{f_i\}_{i=1}^{k_1},\{g_j\}_{j=1}^{k_2}):=\sum_{n=0}^{\infty}2^{-n}\Big|\sum_{i=1}^{k_1}\int_{X_i}h_n\circ F_iJ_{f_i}-\sum_{j=1}^{k_2}\int_{Y_j}h_n\circ G_jJ_{g_j}\Big|.    
		\end{equation}
		Since $M$ is a closed manifold isometrically embedded in $\mathbb{R}^N$. We can define varifold distance using  $\mathbb{R}^N$ instead of $M$. Namely, let \[P_{\tilde{\Pi}}:\mathcal{G}_k\mathbb{R}^N\to\mathbb{R}^N\]
		be the Grassmanian bundle of k-planes. We can choose a countable dense subset $\{\tilde{h}_n\}$ of all continuous functions on $\mathcal{G}_k\mathbb{R}^N$ with supremum norm at most one. If
		the pairs $\{X_i,\tilde{F}_i\}$ and $\{Y_j,\tilde{G}_j\}$ represent compact surfaces $\{X_i,Y_j\}$ of dimension $k$ with measurable maps 
		$\tilde{F}_i:X_i\to \mathcal{G}_k\mathbb{R}^N$ and $\tilde{G}_j:Y_j\to \mathcal{G}_k\mathbb{R}^N$
		, so that  
		\begin{align*}
		f_i&=\tilde{\Pi}\circ\tilde{F}_i,\\
		g_j&=\tilde{\Pi}\circ\tilde{G}_j.
		\end{align*}
		Then the varifold distance between them is defined by
		\begin{equation}
		d_{V_R}(\{f_i\}_{i=1}^{k_1},\{g_j\}_{j=1}^{k_2}):=\sum_{n=0}^{\infty}2^{-n}\Big|\sum_{i=1}^{k_1}\int_{X_i}\tilde{h}_n\circ \tilde{F}_iJ_{f_i}-\sum_{j=1}^{k_2}\int_{Y_j}\tilde{h}_n\circ \tilde{G}_jJ_{g_j}\Big|.   \end{equation} 
	\end{definition}
\begin{remark}\label{varifold energy}
We can assume $h_0$ is constant $1$ in definition \ref{varifold distance}. Given two maps $u,v:S^2\to M$ and $u,v\in W^{1,2}(S^2,M)\cap C^0(S^2,M)$. If $d_{V}(u,v)=0$, then it's easy to see by (\ref{2.49}), we have
\[\text{Area}(u)=\text{Area}(v),\:\text{and }E(u)=E(v).\]
\end{remark}
\begin{proposition}
	[Colding-Minicozzi, \cite{CD}]\label{bv}
		If a sequence $\{\gamma_j\}$ of $W^{1,2}(S^2,M)$ maps bubble converges to a collection of finitely many smooth maps $u_0,...,u_n:S^2\to M$ then it also varifold converges to $u_0,...,u_n$.
\end{proposition}
\begin{proof}
Let $E=\sum_{i=0}^nE(u_i)$, since $\gamma_j$ bubble converges to $\{u_i\}_{i=0}^n$, we assume without loss of generality that $d_B(\gamma_j,\{u_i\}_{i=0}^n)<1/j$. So there exists conformal dilations $D^j_{i}\in PSL(2,\mathbb{C}),\:i=0,...,n,$ and pairwise disjoint domains $\Omega^j_{0},...,\Omega^j_{n}$, $\bigcup_{i=0}^n\Omega^j_{i}\subset S^2.$ such that the following holds:
\begin{equation}
		\sum_{i=0}^n\Big(\int_{\Omega^j_{i}}|\nabla(\gamma_j-\nabla (u_i\circ D^j_{i})|^2\Big)^{1/2}<1/j,
		\end{equation}
		\begin{equation}\label{3.9}
		\int_{S^2\setminus\bigcup_{i=0}^n\Omega^j_{i}}|\nabla\gamma_j|^2<3/j^2+2(n+1)E/j,
		\end{equation}
		\begin{equation}
		\sum_{i=0}^n\Big(\int_{S^2\setminus\Omega^j_{i}}|\nabla(u_i\circ D^j_{i})|^2\Big)^{1/2}<1/j.
		\end{equation}
		
For each $u_i$, let $U_i$ denote the corresponding map to $\mathcal{G}_2M$. Similarly, for each $\gamma_j$, let $R_j$ denote the corresponding map to $\mathcal{G}_2M$. We will also use that the map $\nabla u\to J_u$ is continuous as a map from $L^2$ to $L^1$ and thus area of $u$ is continuous with respect to energy of $u$. (see \cite[proposition A.3]{CD})

The proposition now follows by showing for each $i$ and any $h\in C^0(\mathcal{G}_2M)$ that
\begin{align*}
    \sum_{i=0}^n\int_{S^2}h\circ U_i J_{u_i}&=\sum_{i=0}^n\lim_{j\to\infty}\int_{\Omega^j_{i}}h\circ U_i\circ D^j_i J_{u_i\circ D^j_i}\\
    &=\sum_{i=0}^n\lim_{j\to\infty}\int_{\Omega^j_{i}} h\circ R_j J_{\gamma_j}\\
    &=\lim_{j\to\infty}\int_{\cup_i\Omega^j_{i}}h\circ R_j J_{\gamma_j}\\
    &=\int_{S^2}h\circ R_j J_{\gamma_j},
\end{align*}
where the first equality is simply the change of variables formula for integration, and the last equality follows from (\ref{3.9}).
\end{proof}
	Given a collection of harmonic spheres $\{u_i\}_{i=0}^n$. Since $u_i:S^2\to M$, $u_i\in W^{1,2}(S^2,M)$ for each $i$, and $M$ is a closed Riemannian manifold isometrically embedded in $\mathbb{R}^N$, we have that  
	\[\sum_{i=0}^nu_i\circ\phi_i:S^2\to\mathbb{R}^N,\quad\sum_{i=0}^n u_i\circ\phi_i\in W^{1,2}(S^2,\mathbb{R}^N),\]
	for all $\phi_i\in PSL(S,\mathbb{C})$.
\begin{claim}\label{infw12}
For a map $\gamma\in W^{1,2}(S^2,M)$ with $d_B(\gamma,\{u_i\}_{i=0}^n)<\epsilon,$ $\epsilon<1$,
We have the following inequality
\[\inf\Big\{\int_{S^2}|\nabla\gamma-\nabla\big(\sum_{i=0}^nu_i\circ\phi_i\big)|^2\Big|\phi_i\in PSL(2,\mathbb{C})\Big\}<C(n)\epsilon,\]
for some constant $C(n)$ depends on $\{u_i\}_{i=0}^n$.
\end{claim}
\begin{proof}
It follows from for any $\{\phi_i\}_{i=0}^n\in PSL(2,\mathbb{C})$ and pairewise disjoint domains $\{\Omega_i\}_{i=0}^n\subset S^2$ we have the following inequality:
\begin{align*}
\int_{S^2}|\nabla\gamma-\nabla\big(\sum_{i=0}^nu_i\circ\phi_i\big)|^2\leq&\sum_{i=0}^n\Big(\int_{\Omega_{i}}|\nabla(\gamma-\nabla (u_i\circ \phi_{i})|^2\Big)\\
&+\int_{S^2\setminus\bigcup_{i=0}^n\Omega_{i}}|\nabla\gamma|^2\\
&+\sum_{i=0}^n\Big(\int_{S^2\setminus\Omega_{i}}|\nabla(u_i\circ \phi_{i})|^2\Big)\\
\end{align*}
By $d_B(\gamma,\{u_i\}_{i=0}^n)<\epsilon$, we know that there exist pairwise disjoint domains $\{\tilde{\Omega}_i\}_{i=0}^n\subset S^2$, and conformal transformations $\{\tilde{\phi}_i\}_{i=0}^n\in PSL(2,\mathbb{C})$ such that
\begin{align*}
\sum_{i=0}^n\Big(\int_{\tilde{\Omega}_{i}}|\nabla(\gamma-\nabla (u_i\circ \tilde{\phi}_{i})|^2\Big)&+\int_{S^2\setminus\bigcup_{i=0}^n\tilde{\Omega}_{i}}|\nabla\gamma|^2\\
&+\sum_{i=0}^n\Big(\int_{S^2\setminus\tilde{\Omega}_{i}}|\nabla(u_i\circ \tilde{\phi}_{i})|^2\Big)\\
&<2\epsilon^2+3\epsilon^2+2(n+1)\epsilon E.
\end{align*}
Which implies the desired result:
\[\inf\Big\{\int_{S^2}|\nabla\gamma-\nabla\big(\sum_{i=0}^nu_i\circ\phi_i\big)|^2\Big|\phi_i\in PSL(2,\mathbb{C})\Big\}<C(n)\epsilon.\]
Here $C(n)=5+2(n+1)E$.
\end{proof}
Now we know that by Theorem \ref{bv} bubble convergence implies varifold convergence, but the converse is not true, since varifold convergence simply implies that the images of two maps is close in measure sense. The following Proposition states that varifold convergence implies the term $d_B(\cdot,\cdot)$ goes to zero if the term:
\[\inf\Big\{\int_{S^2}|\nabla\gamma-\nabla\big(\sum_{i=0}^nu_i\circ\phi_i\big)|^2\Big|\phi_i\in PSL(2,\mathbb{C})\Big\}\]
goes to $0$.
	\begin{proposition}\label{for unstable}
		Given a collection of finitely many harmonic spheres $\{u_i\}_{i=0}^n$, for all $\epsilon>0$, there exists $\delta>0$ so that if $\gamma\in W^{1,2}(S^2,M)$ satisfies the following conditions:
		\[\inf\Big\{\int_{S^2}|\nabla\gamma-\nabla\big(\sum_{i=0}^nu_i\circ\phi_i\big)|^2\Big|\phi_i\in PSL(2,\mathbb{C})\Big\}<\delta,\]
		and
		\[d_{V_R}(\gamma,\{u_i\}_{i=0}^n)<\delta,\]
		we have 
		\[d_{B}(\gamma,\{u_i\}_{i=0}^n)<\epsilon.\]
	\end{proposition}
	\begin{proof}
		For the case that $\{u_i\}_{i=0}^n$ contains only one harmonic sphere say $u_0$, then given $\epsilon>0$ and let $\delta=\epsilon/2$. The condition
		\[\inf\Big\{\int_{S^2}|\nabla\gamma-\nabla\big(u_0\circ\phi\big)|^2\Big|\phi\in PSL(2,\mathbb{C})\Big\}<\delta\]
		implies that there exists $\phi_0\in PSL(2,\mathbb{C})$ such that $\int_{S^2}|\nabla\gamma-\nabla\big(u_0\circ\phi_0\big)|^2<\epsilon$, 
		thus we have $d_{B}(\gamma,u_0)<\epsilon.$

		For the general case it suffices to consider that $\{u_i\}_{i=0}^n=\{u_0,u_1\}$. We argue by contradiction, suppose there exists  $\epsilon>0$ and a sequence $\{\gamma_j\}_{j\in\mathbb{N}}$, $\gamma_j\in W^{1,2}(S^2,M)$, such that \[\inf\Big\{\int_{S^2}|\nabla\gamma_j-\nabla\big(\sum_{i=0,1}u_i\circ\phi_i\big)|^2\Big|\phi_i\in PSL(2,\mathbb{C})\Big\}<1/j
		,\:\text{and }d_{V_R}(\gamma_j,\{u_i\}_{i=0,1})<1/j,\]
		but does not imply $d_B(\gamma_j,
		\{u_0,u_1\})<\epsilon$ for all $j\in\mathbb{N}$. That means for any parewise disjoint domains $\Omega_i\subset S^2$, $\Omega_0\cap\Omega_1=\{\emptyset\}$, and conformal automorphisms $\phi_i\in PSL(2,\mathbb{C}), \:i=0,1$  with $\sum_{i=0,1}\Big(\int_{S^2\setminus\Omega_i}|\nabla (u_i\circ\phi_i)|^2\Big)^{1/2}<\epsilon,$ we have either
		\begin{equation}\label{1158}
		\int_{S^2\setminus\bigcup_i\Omega_i}|\nabla\gamma_j|^2\geq 3\epsilon^2+4\epsilon E,
		\end{equation}
		where $E=E(u_0)+E(u_1)$, or
		\begin{equation}\label{1358}
		\sum_{i=0,1}\Big(\int_{\Omega_i}|\nabla\gamma_j-\nabla(u_i\circ\phi_i)|^2\Big)^{1/2}\geq\epsilon.
		\end{equation}
		(see definition \ref{bubble norm}). We first consider the case of (\ref{1158}) assuming the following condition $\sum_{i=0,1}\Big(\int_{\Omega_i}|\nabla\gamma_j-\nabla(u_i\circ\phi_i)|^2\Big)^{1/2}<\epsilon.$ We consider the following inequality
		\begin{align*}
		\int_{S^2}|\nabla u_0|^2+\int_{S^2}|\nabla u_1|^2=\sum_{i=0,1}&\Big(\int_{\Omega_i}|\nabla (u_i\circ\phi)|^2+\int_{S^2\setminus\Omega_i}|\nabla(u_i\circ\phi)|^2\Big)\\
		<\sum_{i=0,1}&\Big(\int_{\Omega_i}|\nabla\gamma_j|^2\\
		&+2\big(\int_{\Omega_i}|\nabla\gamma_j|^2\big)^{1/2}\big(\int_{\Omega_i}|\nabla\gamma_j-\nabla( u_i\circ\phi)|^2\big)^{1/2}\\
		&+\int_{\Omega_i}|\nabla\gamma_j-\nabla (u_i\circ\phi)|^2+\int_{S^2\setminus\Omega_i}|\nabla(u_i\circ\phi)|^2\Big)\\
		<\sum_{i=0,1}&\int_{\Omega_i}|\nabla\gamma_j|^2+2\epsilon\big(\int_{\Omega_i}|\nabla\gamma_j|^2\big)^{1/2}+\epsilon^2+\epsilon^2.
		\end{align*}
Since $d_{V_R}(\gamma_j,\{u_i\}_{i=0,1})<1/j$,  we know that 
\[\lim_{j\to\infty}E(\gamma_j)=E(u_0)+E(u_1),\]
(see remark \ref{varifold energy}.)
For $j$ sufficiently large, we have that \begin{align*}
E(\gamma_j)&=\int_{\bigcup_i\Omega_i}|\nabla\gamma_j|^2+\int_{S^2\setminus\bigcup_i\Omega_i}|\nabla\gamma_j|^2\\
&<\sum_{i=0,1}\int_{\Omega_i}|\nabla\gamma_j|^2+2\epsilon\big(\int_{\Omega_i}|\nabla\gamma_j|^2\big)^{1/2}+\epsilon^2+\epsilon^2\\
&<\sum_{i=0,1}\int_{\Omega_i}|\nabla\gamma_j|^2+2\epsilon\big(E(u_0)+E(u_1)\big)^{1/2}+2\epsilon^2.
\end{align*}
which implies that
\[\int_{S^2\setminus\bigcup_i\Omega_i}|\nabla\gamma_j|^2< 4\epsilon(E(u_0)+E(u_1))^{1/2}+ 2\epsilon^2,\]
contradicting our assumption (\ref{1158}). 

Now we consider the case of ($\ref{1358}$). By our assumption of the sequence $\{\gamma_j\}_{j\in\mathbb{N}}$: $\inf\Big\{\int_{S^2}|\nabla\gamma_j-\nabla\big(\sum_{i=0,1}u_i\circ\phi_i\big)|^2\Big|\phi_i\in PSL(2,\mathbb{C})\Big\}<1/j$, we know that for each $j$ there exist $\phi_{0}^j,\phi_1^j$ such that
\[\int_{S^2}|\nabla\gamma_j-\nabla(u_0\circ\phi_0^j+u_1\circ\phi_1^j)|^2<2/j,\]
which implies that:
\begin{equation}\label{a}
\begin{split}
		d_{V_R}\big((u_0\circ\phi_0^j+u_1\circ\phi_1^j),\{u_i\}_{i=0,1}\big)\leq& d_{V_R}(u_0\circ\phi_0^j+u_1\circ\phi_1^j,
		\gamma_j)\\
		&+d_{V_R}(\gamma_j,\{u_i\}_{i=0,1})\\
		<&C(j)+1/j,    
		\end{split}
		\end{equation}
		for some $C(j)\to 0$ as $j\to\infty$. Since $(u_0\circ\phi_0^j+u_1\circ\phi_1^j)(S^2)=(u_0\circ(\phi_0^j\circ(\phi_1^j)^{-1})+u_1)(S^2)$, we can assume $\phi_1^j$ is identity for all $j\in\mathbb{N}$. Thus by (\ref{a}) we know that $(u_0\circ\phi_0^j+u_1)(S^2)$ can be arbitrarily close to $\{u_0(S^2),u_1(S^2)\}$ as $j\to\infty$ in measure sense in $\mathbb{R}^N$ (see definition \ref{varifold distance}), this implies that $\{\phi_0^j\}_{j\in\mathbb{N}}$ is a divergent sequence in $PSL(2,\mathbb{C})$. Now we choose $R>0$  so that for $R'\geq R$ we have the following \[\Big(\int_{S^2\setminus B_{R'}(p)}|\nabla u_1|^2\Big)^{1/2}<\epsilon/5,\:\forall p\in S^2.\]
		Since $\{\phi_0^j\}_{j\in\mathbb{N}}$ is diverging, then for $j$ sufficiently large we can find $q\in S^2$ such that
		\begin{equation}\label{2.57}
		\int_{B_R(q)}|\nabla(u_0\circ\phi_0^j)|^2+\int_{S^2\setminus B_R(q)}|\nabla u_1|^2<\epsilon^2/10, \end{equation}
		 Then the assumption (\ref{1358}) implies
		\begin{equation}\label{2.56}
	\Big(\int_{B_R(q)}|\nabla\gamma_j-\nabla u_1|^2\Big)^{1/2}+\Big(\int_{S^2\setminus B_R(q)}|\nabla\gamma_j-\nabla(u_0\circ\phi_0^j)|^2\Big)^{1/2}\geq\epsilon.   
		\end{equation}
		From the inequality
	\begin{align*}
	\int_{S^2}|\nabla\gamma_j-\nabla(u_0\circ\phi_0^j+u_1)|^2&\geq \Big|\Big(\int_{B_R(q)}|\nabla\gamma_j-\nabla u_1|^2+\int_{S^2\setminus B_R(q)}|\nabla\gamma_j-\nabla(u_0\circ\phi_0^j)|^2\Big)\\
	   &-\Big(\int_{B_R(q)}|\nabla(u_0\circ\phi_0^j)|^2+\int_{S^2\setminus B_R(q)}|\nabla u_1|^2\Big)\Big|,
\end{align*}
with the assumption  $\int_{S^2}|\nabla\gamma_j-\nabla(u_0\circ\phi_0^j+u_1)|^2<2/j $ and (\ref{2.56}) we have the following:
\begin{align*}
 \int_{B_{R}(q)}|\nabla(u_0\circ\phi_0^j)|^2+\int_{S^2\setminus B_R(q)}|\nabla u_1|^2&\geq\int_{B_R(q)}|\nabla\gamma_j-\nabla u_1|^2\\
 &+\int_{S^2\setminus B_R(q)}|\nabla\gamma_j-\nabla(u_0\circ\phi_0^j)|^2\\
 &-\int_{S^2}|\nabla\gamma_j-\nabla(u_0\circ\phi_0^j+u_1)|^2\\
 &\geq \epsilon^2/5,
\end{align*}
for $j$ sufficiently large. It contradicts with the assumption (\ref{2.57}). Thus we have proved proposition \ref{for unstable}.
	\end{proof}
	\begin{remark}
Combining claim \ref{infw12} and proposition \ref{for unstable} we have for all $\epsilon>0$, there exists $\delta>0$ such that 
\[d_{B}(\gamma,\{u_i\}_{i=0}^n)<\epsilon,\]
if and only if
		\[\inf\Big\{\int_{S^2}|\nabla\gamma-\nabla\big(\sum_{i=0}^nu_i\circ\phi_i\big)|^2\Big|\phi_i\in PSL(2,\mathbb{C})\Big\}<\delta,\]
		and
		\[d_{V_R}(\gamma,\{u_i\}_{i=0}^n)<\delta.\]
	\end{remark}
\subsection{Statement of Colding and Minicozzi's Min-max thoery}
We state some basic notations and min-max theorem in this section.
\begin{definition}[Width]\label{width}
Let $\Omega$ be the set of continuous maps $\sigma:S^2\times [0,1] \to M$ so that for each $t\in[0,1]$ the map $\sigma(\cdot,t)$ is in $C^0(S^2,M)\cap W^{1,2}(S^2,M)$, the map $t\to \sigma(\cdot,t)$ is continuous from $[0,1]$ to $C^0(S^2,M)\cap W^{1,2}(S^2,M)$ in a strong sense. Given a map $\beta\in\Omega$, the homotopy class $\Omega_\beta$ is defined to be the set of maps $\sigma\in\Omega$ that is homotopic to $\beta$ through maps in $\Omega$. We'll call any such $\sigma$ a \textit{sweepout}. 
		The width $W=W_E(\beta,M)$ associated to the homotopy class $\Omega_\beta$ is defined by:
		\begin{equation}\label{44.4}
		W:=\inf_{\sigma\in\Omega_\beta}\max_{t\in[0,1]}E(\sigma(\cdot,t)).
		\end{equation}
		We could alternatively define the width using area rather than energy by setting
		\[W_A:=\inf_{\sigma\in\Omega_\beta}\max_{t\in[0,1]}\text{Area}(\sigma(\cdot,t)).\]
\end{definition}

	\begin{remark}
		We're interested in the case where $\beta$ induces a map in a nontrivial class in $\pi_3(M)$, in which case the width is positive.
	\end{remark}	

\begin{definition}[Minimizing sequence]\label{minimizing sequence}
Given a sweepout $\gamma_{j}(\cdot,t):S^2\times[0,1]\to M,$ we call $\{\gamma_{j}(\cdot,\cdot)\}_{j\in\mathbb{N}}$ a \textit{minimizing sequence} if $$\lim\limits_{j\to\infty}\max_{t\in[0,1]}E(\gamma_j(\cdot,t))=W.$$
We call $\{\gamma_j(\cdot,t_j)\}_{j\in\mathbb{N}}$ a \textit{min-max sequence} if 
\[\lim_{j\to\infty}E(\gamma_j(\cdot,t_j))=W.\]
\end{definition}

	\begin{definition}\label{equivalentclass}
		We define the equivalent class of $u:S^2\to M$ to be:
		\[[u]:=\Big\{g:S^2\to M\Big|\text{ if } u=g\circ\phi\text{ for some }\phi\in PSL(2,\mathbb{C})\Big\},\]
	\end{definition}
	\begin{remark}
	Given maps $\gamma,\{u_i\}_{i=0}^n\in W^{1,2}(S^2,M)$ with $d_B(\gamma,\{u_i\}_{i=0}^n)<\epsilon$, we have \[d_B(\gamma,\{g_i\}_{i=0}^n)<\epsilon,\:\text{if }[g_i]=[u_i],\:i=0,...,n.\]
	\end{remark}
\begin{definition}[Image set]\label{image set}
		The image set $\Lambda(\{\gamma_j(\cdot,t)\})$ of $\{\gamma_j(\cdot,t)\}_{j\in\mathbb{N}}$ is defined to be:
		\begin{align*}
		 \Lambda(\{\gamma_j(\cdot,t)\}_{j\in\mathbb{N}}):=\Big\{\{[u_i]\}_{i=0}^n:&\text{there  exists a sequence }\{i_j\}\to\infty, t_{i_j}\in[0,1],\\
		 &\text{such that }\gamma_{i_j}(\cdot,t_{i_j})\text{ bubble converges to }\{u_i\}_{i=0}^n\Big\},
		\end{align*}

\end{definition}
Now we state the min-max theorem for harmonic sphere. Theorem \ref{CDDD} isn't exactly what's stated in \cite{CD}, it uses $d_B(\cdot,\cdot)$ instead of varifold norm and applies to any minimizing sequence. We prove in appendix \ref{F1.section} that Colding-Minicozzi's result \cite{CD} does imply theorem \ref{CDDD}.
\begin{theorem}[Min-Max for harmonic sphere]\label{CDDD}
    Given a closed manifold $M$ with dimension at least three, and a map $\beta\in\Omega$ representing a nontrivial class in $\pi_3(M)$, then for any sequence of sweepouts $\gamma_j\in\Omega_\beta$ with 
    \[\lim_{j\to\infty}\max_{s\in[0,1]}E(\gamma_j(\cdot,s))= W,\] 
    there exists a subsequence $\{i_j\}\to\infty$,  $t_{i_j}\in[0,1]$, and a collection of finitely many harmonic spheres $\{u_i\}_{i=0}^n$ such that
    \[d_B(\gamma_{i_j}(\cdot,t_{i_j}),\{u_i\}_{i=0}^n)<1/j.\]
    \end{theorem}

	\section{Unstable Lemma}

The main focus of the section is lemma $\ref{unstable}$: proving the energy is concave for maps that are sufficiently close to a finite collection of harmonic spheres in bubble tree sense. We first consider the simplest example, for a given map $u\in W^{1,2}(S^2,M)$, and $X\in C^\infty(S^2,\mathbb{R}^N)$ with
\[\delta^2E(u)(X,X)<0.\]
By the form of second variation of energy (see (\ref{2.29})), clearly if $\epsilon>0$ is sufficiently small, then for any $\gamma\in W^{1,2}(S^2,M)$ with $\|\gamma-u\|_{W^{1,2}}<\epsilon$, we have
\[\delta^2E(\gamma)(X,X)<0.\]
Now we consider the general case, given a finite collection of harmonic spheres $\{u_i\}_{i=0}^n$ with $\sum_{i=0}^n \text{Index}(u_i)=k>0$. 
Index assumption implies there are $k$ correspongding vector fields $\{X_l\}_{l=1}^k$, $X_l\in C^{\infty}(S^2,\mathbb{R}^N)$, positive constant $c_l>0$ for each $l$, and the corresponding harmonic spheres $v_l\in\{u_i\}_{i=0}^n$, such that
\[\delta^2 E(v_l)(X_l,X_l)=-c_l<0.\]
Now the goal is to choose $\epsilon>0$ so that for any $\gamma\in W^{1,2}(S^2,M)$ with $d_B(\gamma,\{u_i\}_{i=0}^n)<\epsilon$, we can construct $\{\tilde{X}_l\}_{l=1}^k$, $\tilde{X}_l\in C^\infty(S^2,\mathbb{R}^N)$, then the variation $\gamma_s=\Pi\circ\big(\gamma+s\tilde{X}_l)$ satisfies
\[\frac{d^2}{ds^2}\at[\Big]{s=0}\int_{S^2}|\nabla\gamma_s|^2<\frac{1}{2}\delta^2E(v_l)(X_l,X_l),\]
for each $l$.

The proof of lemma \ref{unstable} is long and detailed but the idea behind it is simple. It can be roughly spoken as the following: if $d_B(\gamma,\{u_i\}_{i=0}^n)<\epsilon$ for some $\gamma\in W^{1,2}(S^2,M)$, since $v_l\in\{u_i\}_{i=0}^n$, there exist $\Omega_l\subset S^2$ and $D_l\in PSL(2,\mathbb{C})$, so that 
\[\int_{\Omega_l}|\nabla\gamma-\nabla (v_l\circ D_l)|^2<\epsilon^2,\]
and 
\begin{equation}\label{6.1}
\int_{S^2\setminus\Omega_l}|\nabla(v_l\circ D_l)|^2<\epsilon^2,
\end{equation}
see definition (\ref{bubble norm}). If $\epsilon$ is sufficiently small the following term is small
\[\frac{d^2}{ds^2}\at[\Big]{s=0}\int_{\Omega_l}|\nabla\Pi\circ(v_l\circ D_l+sX_l\circ D_l)|^2-\frac{d^2}{ds^2}\at[\Big]{s=0}\int_{\Omega_l}|\nabla\Pi\circ(\gamma+sX_l\circ D_l)|^2.\]
By choosing a suitable cutoff function $\eta_l$ we can make the following term small
\[\frac{d^2}{ds^2}\at[\Big]{s=0}\int_{S^2}|\nabla\Pi\circ(\gamma+s\eta_lX_l\circ D_l)|^2-\frac{d^2}{ds^2}\at[\Big]{s=0}\int_{\Omega_l}|\nabla\Pi\circ(\gamma+sX_l\circ D_l)|^2.\]
Let $\tilde{X}_l=\eta_lX_l\circ D_l$ and $\gamma_s=\Pi\circ\big(\gamma+s\tilde{X}_l)$, observe that
\begin{equation}\label{6.2}
\begin{split}
\delta^2E(v_l)(X,X)=&\frac{d^2}{ds^2}\at[\Big]{s=0}\int_{\Omega_l}|\nabla\Pi\circ(v_l\circ D_l+sX\circ D_l)|^2\\
&+\frac{d^2}{ds^2}\at[\Big]{s=0}\int_{S^2\setminus\Omega_l}|\nabla\Pi\circ(v_l\circ D_l+sX\circ D_l)|^2.
\end{split}    
\end{equation}
The first term of the right hand side of (\ref{6.2}) is close to 
$\frac{d^2}{ds^2}\at[\Big]{s=0}\int_{S^2}|\nabla\gamma_s|^2$, and the second term is small because of (\ref{6.1}). We have the desired inequality \[\frac{d^2}{ds^2}\at[\Big]{s=0}\int_{S^2}|\nabla\gamma_s|^2<\frac{1}{2}\delta^2E(v_l)(X_l,X_l).\]
We now state and prove the unstable lemma and specify how to choose $\epsilon>0$ and $\eta_l.$

\begin{lemma}[Unstable lemma]\label{unstable} Let $M$ be a closed manifold of dimension at least three, isometrically embedded in $\mathbb{R}^N$. Given a collection of finitely many harmonic spheres $\{u_i\}_{i=0}^n$ with $\sum_{i=0}^n Index(u_i)=k$. There exist $1>c_0>0$ and $\epsilon>0$, so that if $d_B(\gamma,\{u_i\}_{i=0}^n)<\epsilon$ for $\gamma\in W^{1,2}(S^2,M)$, then we can construct vector fields $\{\tilde{X}_l\}_{l=1}^k$, $\tilde{X}_l\in C^{\infty}(S^2,\mathbb{R}^N)$ for each $l$, define the variation $\gamma_s$ of $\gamma$ as 
\[\gamma_s=\Pi\circ\big(\gamma+\sum_{l=1}^k s_l\tilde{X}_l\big),\quad s=(s_1,...,s_k)\in\bar{B}^k,\]
and let $E_{\gamma}(s):=E(\gamma_s),$ so that the following hold:
\begin{enumerate}
    \item $E_\gamma(s)$ has a unique maximum at $m_\gamma\in B^k_\frac{c_0}{\sqrt{10}}(0).$
    \item The map $\gamma\mapsto m_\gamma$ is continuous.
\item $\forall s\in \bar{B}^k$ we have
		\begin{equation}\label{m}
		-\frac{1}{c_0}Id\leq D^2E_\gamma(s)\leq-c_0Id,\
		\end{equation}
		and
		\begin{equation}\label{mm}
		E_\gamma(m_\gamma)-\frac{1}{2c_0}|m_\gamma-s|^2\leq E_\gamma(s)\leq E(m_\gamma)-\frac{c_0}{2}|m_\gamma-s|^2.
		\end{equation}	
\end{enumerate}
	\end{lemma}
	\begin{proof}
Index assumption implies there are $k$ correspongding vector fields $\{X_l\}_{l=1}^k$, $X_l\in C^{\infty}(S^2,\mathbb{R}^N)$, positive constants $c_l>0$ for each $l=1,...,k$, and the corresponding harmonic spheres $v_l\in\{u_i\}_{i=0}^n$, such that
\begin{equation}
\delta^2 E(v_l)(X_l,X_l)=-c_l<0. \end{equation}
Let $\xi:=c_l/C>0$, $C$ is a constant which will be chosen later. By (\ref{2.29}), there exists $\delta(\xi)>0$ depending on $\{M,X_l\}$ such that
\begin{equation}\label{6.6}
|\delta^2 E(v_l)(X_l,X_l)-\delta^2 E(\gamma)(X_l,X_l)|<\xi,    
\end{equation}
for all $\gamma$ with $\int_{S^2}|\nabla v_l-\nabla\gamma|^2<\delta(\xi)$. We define $v_{l,s}:=\Pi\circ(v_l+sX_l)$. There exists $\rho>0$ such that for all $p\in S^2$ and $\varrho<\rho$ we have:
		\begin{equation}\label{chat}
		-\xi<\frac{d^2}{ds^2}\at[\Big]{s=0}\int_{B_\varrho(p)}|\nabla v_{l,s}|^2<\xi,\quad\text{and }
		\int_{B_\varrho(p)}|\nabla X_l|^2<\xi.
		\end{equation}
We choose $J\in\mathbb{N}$ so that  
		\begin{equation}\label{bonbon}
		-\frac{1}{\log 1/J}<\xi,
		\end{equation}
		and define $\varepsilon_J$ to be $\min\Big\{\int_{B_{1/J}(p)}|\nabla v_l|^2\Big|\:p\in S^2\Big\},\:i\in\mathbb{N},$
		note that $\varepsilon_J$ is strictly positive.
		We now choose $\epsilon>0$ to be the constant satisfying the following inequality \[\max\{\epsilon^2,(3\epsilon^2+2(n+1)\epsilon E)\}<\min\{\varepsilon_J/2,\xi,\delta(\xi)\},\]
		and consider $\gamma\in W^{1,2}(S^2,M)$ with $d_B(\gamma,\{u_i\}_{i=0}^n)<\epsilon$. Since $v_l\in\{u_i\}_{i=0}^n$, there is $\Omega_l\in S^2$ and a conformal dilation $D_l:S^2\to S^2$ such that
		\begin{equation}
		\int_{\Omega_{l}}|\nabla\gamma-\nabla (v_l\circ D_{l})|^2<\epsilon^2,
		\end{equation}
		and
		\begin{equation}\label{coucon}
		\int_{S^2\setminus\Omega_{l}}|\nabla(v_l\circ D_{l})|^2<\epsilon^2.
		\end{equation}
		Moreover, we can choose $\tilde{\Omega}_l$ with $\Omega_l\subset\tilde{\Omega}_l$ so that 
		\[\int_{\tilde{\Omega}_l\setminus\Omega_l}|\nabla\gamma|^2<3\epsilon^2+2(n+1)\epsilon  E,\] 
		here $E=\sum_{i=0}^nE(u_i)$. Assume that $S^2\setminus (D_l\circ\tilde{\Omega}_l)$ and $S^2\setminus (D_l\circ\Omega_l)$ are geodesic balls which center at some point $p\in S^2$, namely, $S^2\setminus (D_l\circ\Omega_l)=B_r(p)$ and $S^2\setminus (D_l\circ\tilde{\Omega}_l)=B_{r^k}(p)$ for some $1<k\leq 2$. By equation \ref{coucon} we know that \[\int_{B_r(p)}|\nabla v_l|^2<\varepsilon_J/2,\] and it implies that $r$ must be smaller than $1/J$.

		Now we define the following piecewise smooth cutoff function, which was introduced by Choi and Schoen \cite{CHS},  $\eta:[0,\infty)\to[0,1]$:
		\[\eta(x) =\left\{ \begin{array}{rcl}
		0, & \mbox{for}	& x<r^k, \\
		(k+1)-(\log x)/(\log r),   & \mbox{for} & r^k\leq x \leq r,\\
		1, & \mbox{for} & x>r, 
		\end{array}\right.\]
		so that
		\[\frac{d\eta}{dx}(x) =\left\{ \begin{array}{rcl}
		0, & \mbox{for}	& x<r^k, \\
		-1/x(\log r),   & \mbox{for} & r^k\leq x \leq r,\\
		0, & \mbox{for} & x>r, 
		\end{array}\right.\]
		and 
		\[\int_{0}^{2\pi}\int_{r^k}^{r}\Big(\frac{d\eta}{dx}(x)\Big)^2xdxd\theta=-\frac{2\pi(k-1)}{\log r}.\]
		Since we have $r<1/J$, (\ref{bonbon}) implies that
		\[-\frac{2\pi(k-1)}{\log r}<2\pi(k-1)\xi.\]
		Now we define $\eta_l=\eta\circ y_l:S^2\to [0,1]$, where $y_l:S^2\to[0,\infty)$ and $y_l(q)=x$ for $q\in\partial B_x(p)$. $\eta_l$ is compactly supported in $S^2\setminus B_{r^k}(p)$ and has value 1 in $S^2\setminus B_{r}(p)$, then
		\begin{equation}\label{cutoff}
		\int_{S^2}|\nabla\eta_l|^2<2\pi(k-1)\xi.    
		\end{equation}
		Let $\tilde{X}'_l:=(\eta_l\circ D_l)X_l\circ D_l,$ and define $v_{l,s}=\Pi\circ(v_l+sX_l)$,  $\gamma_{s}:=\Pi\circ(\gamma+s\tilde{X}_l')$. Then we have:
		\begin{equation}\label{poison}
		\begin{split}
		\Big|\delta^2E(\gamma)(\tilde{X}'_l,\tilde{X}'_l)-\delta^2E(v_l)(X_l,X_l)\Big|&=\Big|\frac{d^2}{ds^2}\at[\Big]{s=0}\int_{S^2}|\nabla\gamma_s|^2-\frac{d^2}{ds^2}\at[\Big]{s=0}\int_{S^2}|\nabla v_{l,s}|^2\Big|\\
		&<\Big|\frac{d^2}{ds^2}\at[\Big]{s=0}\int_{\Omega_l}|\nabla\gamma_{s}|^2-\frac{d^2}{ds^2}\at[\Big]{s=0}\int_{D_l\circ\Omega_l}|\nabla v_{l,s}|^2\Big|\\
		&+\Big|\frac{d^2}{ds^2}\at[\Big]{s=0}\int_{S^2\setminus(D_l\circ\Omega_l)}|\nabla v_{l,s}|^2\Big|\\
		&+\Big|\frac{d^2}{ds^2}\at[\Big]{s=0}\int_{S^2\setminus\Omega_l}|\nabla\gamma_{s}|^2\Big|.
		\end{split}
		\end{equation}
Since $\int_{\Omega_l}|\nabla v_l-\nabla\gamma|^2<\epsilon^2<\delta(\xi)$, (\ref{6.6}) implies that
\begin{equation}\label{6.13}
		\Big|\frac{d^2}{ds^2}\at[\Big]{s=0}\Big(\int_{\Omega_l}|\nabla\gamma_{s}|^2-|\nabla (v_{l,s}\circ D_l)|^2\Big)\Big|
		  <\xi,
	\end{equation}
Since $S^2\setminus(D_l\circ\Omega_l)=B_r(p)$ by the assumption, the choice of $\epsilon$ implies that $r<\rho$ and equation \ref{chat} implies that
\begin{equation}\label{6.14}
\Big|\frac{d^2}{ds^2}\at[\Big]{s=0}\int_{S^2\setminus(D_l\circ\Omega_l)}|\nabla v_{l,s}|^2\Big|<\xi. 
\end{equation}
Now we consider the last term of equation \ref{poison}, namely: \[\Big|\frac{d^2}{ds^2}\at[\Big]{s=0}\int_{S^2\setminus\Omega_l}|\nabla\gamma_{s}|^2\Big|,\]
by equation (\ref{2.29}) we have:
		\begin{equation}\label{estimate}
		\begin{split}
		\Big|\frac{d^2}{ds^2}\at[\Big]{s=0}\int_{S^2\setminus\Omega_l}|\nabla\gamma_{s}|^2\Big|\leq&\int_{S^2\setminus \Omega_l}|d\Pi_\gamma(\nabla\tilde{X}'_l)|^2+C_1\int_{S^2\setminus\Omega_l}\big(|\nabla\tilde{X}_l'||\nabla\gamma|+|\tilde{X}_l'|^2|\nabla\gamma|^2\big)\\
		\leq&\int_{\tilde{\Omega}_l\setminus\Omega_l}|\nabla\tilde{X}_l'|^2\\
		&+C_1\Big(\int_{\tilde{\Omega}_l\setminus\Omega_l}|\nabla\tilde{X}_l'|^2\Big)^{1/2}\Big(\int_{\tilde{\Omega}_l\setminus\Omega_l}|\nabla\gamma|^2\Big)^{1/2}\\
		&+C_1\sup_{p\in S^2}|X_l(p)|^2\int_{\tilde{\Omega}_l\setminus\Omega_l}|\nabla\gamma|^2,
		\end{split}
		\end{equation}
		here $C_1$ is a constant which depends on $M$ (since $\text{Hess}\Pi_\gamma(\cdot,\cdot)$ is bounded by the second fundamental form of $M$ \cite[Appendix 2.12]{LS} and $\nabla\text{Hess}\Pi$ is bounded by curvature of $M$). Then 
		\begin{equation}\label{26}
		\begin{split}
		\int_{\tilde{\Omega}_l\setminus\Omega_l}|\nabla\tilde{X}_l'|^2=&\int_{B_{r}(p)\setminus B_{r^k}(p)}|\nabla(\eta_l X_l)|^2\\
		=&\int_{B_{r}(p)\setminus B_{r^k}(p)}|(\nabla\eta_l)X_l+\eta_l\nabla X_l|^2\\
		\leq&\int_{B_{r}(p)\setminus B_{r^k}(p)}|\nabla X_l|^2\\
		&+2\sup_{p\in S^2}|X_l(p)|^2\Big(\int_{B_{r}(p)\setminus B_{r^k}(p)}|\nabla X_l|^2\Big)^{1/2}\Big(\int_{B_{r}(p)\setminus B_{r^k}(p)}|\nabla\eta_l|^2\Big)^{1/2}\\
		&+\sup_{p\in S^2}|X_l(p)|^2\int_{B_{r}(p)\setminus B_{r^k}(p)}|\nabla\eta_l|^2\\
		<&C_2\xi,
		\end{split}
		\end{equation}
		for some constant $C_2(M,X_l)$, the last inequality follows from equation (\ref{cutoff}) and (\ref{chat}). By  (\ref{26}) we can bound  (\ref{estimate}) by: 
		\begin{equation}\label{6.17}
		\Big|\frac{d^2}{ds^2}\at[\Big]{s=0}\int_{S^2\setminus\Omega_l}|\nabla\gamma_{s}|^2\Big|<C_2\xi+C_1\sqrt{C_2}\xi+C_1\sup_{p\in S^2}|X_l(p)|^2\xi<C_3\xi.
		\end{equation}
Finally, combining the inequality (\ref{6.13}), (\ref{6.14}), and (\ref{6.17}) we have that (\ref{poison}) is bounded by
		\begin{equation}
		\Big|\delta^2E(\gamma)(\tilde{X}'_l,\tilde{X}'_l)-\delta^2E(v_l)(X_l,X_l)\Big|<\xi+\xi+C_3\xi.
		\end{equation}
		Since $\xi=c_l/C$, we now pick $C$ to be a constant strictly larger than $5(C_3+2)$, then we have
		\[-\frac{6}{5}c_l<\frac{d^2}{ds^2}\at[\Big]{s=0}E_{\gamma}(s)<-\frac{4}{5}c_l,\]
		since $\frac{d^2}{ds^2}\at[\Big]{s=t}E_{\gamma}(s)$ 
		is continuous with respect to $t$, there exists $\kappa_l(M,X_l,u_l,\epsilon)>0$ such that
		\[-\frac{6}{5}c_l<\frac{d^2}{ds^2}\at[\Big]{s=t}E_{\gamma}(s)<-\frac{4}{5}c_l,\quad\text{for all }t\in[-\kappa_l,\kappa_l].\]
		We can choose a constant $a(\kappa_l)>0$, let $\tilde{X}_l:=a(\kappa_l)(\eta_l\circ D_l)X_l\circ D_l,$ and redefine the variation of $\gamma$ to be $\gamma_{s}:=\Pi\circ(\gamma+s\tilde{X}_l)$ so that
		\[-a^2(\kappa_l)\frac{3}{2}c_l<\frac{d^2}{ds^2}\at[\Big]{s=t}E_{\gamma}(s)<-a^2(\kappa_l)\frac{1}{2}c_l,\quad\text{for all }t\in[-1,1].\]
		We finish the proof by choosing a constant $0<c_0<1$ such that $c_0<\min_{l=1,..,k}a^2(\kappa_l)\frac{1}{2}c_l$ and $1/c_0>\min_{l=1,...,k}a^2(\kappa_l)\frac{3}{2}c_l$.
	
	\end{proof}
The vector fields $\{\tilde{X}_l\}_{l=1}^k$ constructed in lemma \ref{unstable} depend on $\gamma$, the following corollary shows that there exists $\delta_\gamma>0$ such that for all maps $\sigma$ with $\|\sigma-\gamma\|_{W^{1,2}}<\delta_\gamma$. The variation of $\sigma$ with respect to vector fields $\{\tilde{X}_l\}_{l=1}^k$ still satisfies (\ref{m}) and (\ref{mm}) in lemma \ref{unstable}.
	
	\begin{corollary}\label{forrunstable}
Let $M$ be a closed manifold of dimension at least three, isometrically embedded in $\mathbb{R}^N$. Given a collection of finitely many harmonic spheres $\{u_i\}_{i=0}^n$ with $\sum_{i=0}^n Index(u_i)=k$. Let $1>c_0>0$ and $\epsilon>0$ be given as lemma \ref{unstable}. For a map $\gamma\in W^{1,2}(S^2,M)$ with $d_B(\gamma,\{u_i\}_{i=0}^n)<\epsilon$, let $\{\tilde{X}_l\}_{l=1}^k$ be vector fields given as lemma \ref{unstable}.
For $\sigma\in W^{1,2}(S^2,M)$, we define \[\sigma_s:=\Pi\circ(\sigma+\sum_{l=1}^ks_l\tilde{X}_l)\quad\text{for }s=(s_1,...,s_k)\in\bar{B}^k,\] and let $E_\sigma(s):=E(\sigma_s)$, $m_\sigma$ be the maximum of $E_\sigma(s)$.
There exists $\delta_\gamma>0$, such that for $\sigma\in W^{1,2}(S^2,M)$ satisfying
\[\int_{S^2}|\nabla\gamma-\nabla\sigma|^2<\delta_\gamma,\]
the following properties hold
\begin{enumerate}
    \item $E_\sigma(s)$ has a unique maximum at $m_\sigma\in B^k_\frac{c_0}{\sqrt{10}}(0).$
    \item The map $\sigma\mapsto m_\sigma$ is continuous.
\item $\forall s\in \bar{B}^k$ we have
		\begin{equation}
		-\frac{1}{c_0}Id\leq D^2E_\sigma(s)\leq-c_0Id,\
		\end{equation}
		and
		\begin{equation}
		E_\sigma(m_\sigma)-\frac{1}{2c_0}|m_\sigma-s|^2\leq E_\sigma(s)\leq E(m_\sigma)-\frac{c_0}{2}|m_\sigma-s|^2.
		\end{equation}	
\end{enumerate}
	\end{corollary}
	\begin{proof}
	By lemma \ref{unstable}, 
	$\forall s\in \bar{B}^k$, we have
		\begin{equation}
		-\frac{1}{c_0}Id\leq D^2E_\gamma(s)\leq-c_0Id.
		\end{equation}
For each $s\in \bar{B}^k$, there exists $\delta(s)>0$ so that for all $\sigma\in W^{1,2}(S^2,M)$ with 
\[\int_{S^2}|\nabla\sigma-\nabla\gamma|^2<\delta(s),\]
we have
\begin{equation}
		-\frac{1}{c_0}Id\leq D^2E_\sigma(s)\leq-c_0Id.
\end{equation}
Let $\delta_\gamma:=\min_{s\in\Bar{B}^k}\delta(s)$
\begin{claim}
 $\delta_\gamma>0$.
\end{claim}
\begin{proof}[Proof of the claim]
If not, there exists a sequence $\{s_i\}_{i\in\mathbb{N}}$ such that $\lim_{i\to\infty}\delta(t_i)=0$. Since $\bar{B}^k$ is compact, we have that $\lim_{i\to\infty}s_i=s'\in\bar{B}^k$,  and $\delta(s')>0$ implies the desired contradiction.
		\end{proof}
Thus $\forall\sigma\in W^{1,2}(S^2,M)$ with 
\[\int_{S^2}|\nabla\sigma-\nabla\gamma|^2<\delta_\gamma,\]
we have
\begin{equation}
		-\frac{1}{c_0}Id\leq D^2E_\sigma(s)\leq-c_0Id,\quad\forall s\in\bar{B}^k.
\end{equation}
	\end{proof}

	\section{Deformation Theorem}
Let $M$ be a closed manifold with dimension at least three, isometrically embedded in $\mathbb{R}^N$. Consider a map $\beta\in\Omega$ representing a nontrivial class in $\pi_3(M)$,  let $W$ be the width associated to the homotopy class $\Omega_\beta$ (see definition \ref{width}, (\ref{44.4})), and given a sequence of sweepouts $\gamma_{j}(\cdot,t)\in\Omega_\beta$ which is minimizing, i.e., \[\lim_{j\to\infty}\max_{t\in[0,1]}E(\gamma_j(\cdot,t))=W.\] 
Moreover, let $K=\Big\{\{[k^1_i]\}_{i=0}^{m_1},...,\{[k^{N_k}_i]\}_{i=0}^{m_{N_k}}\Big\}$ be a finite set of finite collection of equivalent classes of harmonic spheres, so there exist a constant $\epsilon_k>0$ and $j_k\in\mathbb{N}$ such that
\[d_B(\gamma_j(\cdot,t),\{k^l_i\}_{i=0}^{m_l})>\epsilon_k,\quad\forall t\in[0,1],\]
for all $j>j_k$, $l=1,...,N_k$.
	\begin{theorem}[Deformation Theorem]\label{deformation}
		As assumed above, given a collection of finitely many harmonic spheres $\{u_i\}_{i=0}^n$
		with $\sum_{i=0}^n \text{Index}(u_i)=k>1$ and $\sum_{i=0}^nE(u_i)=W$. There exists a sequence of sweepouts $\{\gamma_j'(\cdot,t)\}_{j\in\mathbb{N}}$ such that
		\begin{enumerate}
			\item $\gamma'_j(\cdot,t)$ is homotopic to $\gamma_j(\cdot,t)$,
			\item $\{\gamma'_j(\cdot,t)\}_{j\in\mathbb{N}}$ is a minimizing sequence,
			\item there exists $j_k'\in\mathbb{N}$ such that
			\[d_B(\gamma'_j(\cdot,t),\{k^l_i\}_{i=0}^{m_l})>\epsilon_k\quad\text{for }l=1,...,N_k,\forall t\in[0,1],\]
			for all $j>j_k'$.
			\item 
			there exists $\epsilon_J>0$ and $J\in\mathbb{N}$ such that
			\[d_B(\gamma_j'(\cdot,t),\{u_i\}_{i=0}^n)>\epsilon_J,\quad\forall t\in[0,1],\]
			for all $j>J$.
		\end{enumerate}
	\end{theorem}
	\begin{proof}
Assumption $\sum_{i=0}^n\text{Index}(u_i)=k$ implies there are $k$ correspongding vector fields $\{X_l'\}_{l=1}^k$, $X'_l\in C^{\infty}(S^2,\mathbb{R}^N)$, such that
$\delta^2 E(v_l)(X_l',X_l')<0,$ for $v_l\in\{u_i\}_{i=0}^n$. Let $\epsilon>0$ be given as lemma \ref{unstable}. By proposition \ref{for unstable}, there exists $\delta>0$ such that
$d_{B}(\gamma,\{u_i\}_{i=0}^n)<\epsilon$, if
\[\inf\Big\{\int_{S^2}|\nabla\gamma-\nabla\big(\sum_{i=0}^nu_i\circ\phi_i\big)|^2\Big|\phi_i\in PSL(2,\mathbb{C})\Big\}<\delta,\]
		and
\[d_{V_R}(\gamma,\{u_i\}_{i=0}^n)<\delta,\]
for $\gamma\in W^{1,2}(S^2,M).$ We consider the following sets:
		\[I_{j,\delta/2}:=\Big\{t\in[0,1]\Big|\inf\big\{\int_{S^2}|\nabla\gamma_j(\cdot,t)-\sum_{i=0}^n\nabla(u_i\circ D_i)|^2\Big|D_i\in PSL(2,\mathbb{C})\big\}\leq\delta/2\Big\},\]
		and
		\[I'_{j,\delta/2}:=\Big\{t\in[0,1]\big|d_{V_R}(\gamma_j(\cdot,t),\{u_i\}_{i=0}^n)\leq\delta/2\Big\}.\]
Let $U_{j,\delta/2}:=I_{j,\delta}\cap I'_{j,\delta/2}.$ Proposition \ref{for unstable} implies that
$d_B(\gamma_j(\cdot,t),\{u_i\}_{i=0}^n)<\epsilon$, for all $t\in U_{j,\delta/2}$. We define 
\[E_j^t(s,\{Y_l\}_{l=1}^k):=\int_{S^2}|\nabla(\Pi\circ(\gamma_j(\cdot,t)+\sum_{l=1}^ks_lY_l))|^2,\quad s=(s_1,...,s_k)\in\bar{B}^k.\]
For $t_m\in U_{j,\delta/2}$, by lemma \ref{unstable}, we can construct vector fields  $\{\tilde{X}_l(t_m)\}_{l=1}^k$, and the hessian of 
$E_j^{t_m}(s,\{\tilde{X}_l(t_m)\}_{l=1}^k)$ with respect to $s\in\bar{B}^k$, which we denote by $D^2_sE_j^{t_m}(s,\{\tilde{X}_l(t_m)\}$, satisfies 
		\begin{equation}
		-\frac{1}{c_0}Id\leq D^2_sE_j^{t_m}(s,\{\tilde{X}_l(t_m)\}_{l=1}^k)\leq-c_0Id,\:\forall s\in\bar{B}^k
		\end{equation}
here $c_0$ is a constant given by lemma \ref{unstable}. By corollary \ref{forrunstable} and the continuity of $\gamma_j(\cdot,t)$ in $W^{1,2}(S^2,M)$ with respect to $t$, we know there exists $\delta(t_m)>0$ so that for all $ t\in(t_m-\delta(t_m),t_m+\delta(t_m))\cap U_{j,\delta/2}$ we have
\begin{equation}
-\frac{1}{c_0}Id\leq D^2_sE_j^{t}(s,\{\tilde{X}_l(t_m)\}_{l=1}^k)\leq-c_0Id,\:\forall s\in\bar{B}^k.
\end{equation}
Let $I^{t_m}:=(t_m-\delta(t_m),t_m+\delta(t_m))\cap U_{j,\delta/2}$, since $U_{j,\delta/2}$ is compact we can cover $U_{j,\delta/2}$ by finitely many $I^{t}$, say $I^{t_1},...,I^{t_{N_1}}$. Moreover, after discarding some of the intervals, we can arrange that each $t$ is in at least one closed interval $\Bar{I}^{t_m}$, each $\Bar{I}^{t_m}$ intersects at most two other $\Bar{I}^{t_k}$'s, and the $\Bar{I}^{t_k}$'s intersecting $\Bar{I}^{t_m}$ do not intersect each other. For each $m=1,...,N_1$, choose a smooth function $\xi_m(t):[0,1]\to[0,1]$ which is supported in $\Bar{I}^{t_m}$, and \[\sum_{m=1}^{N_1}\xi_m(t)=1,\:\forall t\in[0,1].\] 
We define $X_l(t)$ to be
\begin{equation}\label{6.226}
X_l(t):=\sum_{m=1}^{N_1}\xi_m(t)\tilde{X}_l(t_m),\quad t\in[0,1],l=1,...,k.    
\end{equation}
Consider $E_j^{t}(s,\{{X}_l(t)\}_{l=1}^k)$, if $X_l(t)=\Tilde{X}_l(t_m)$ for some $t\in U_{j,\delta/2}$, then obviously we have $-\frac{1}{c_0}Id\leq D^2_sE_j^{t}(s,\{{X}_l(t)\}_{l=1}^k)\leq-c_0Id,\:\forall s\in\bar{B}^k.$ If $X_l(t)=\delta_a(t)\Tilde{X}_l(t_a)+\delta_b(t)\Tilde{X}_l(t_b)$, since $\delta_a(t)+\delta_b(t)=1$, we have for all $s\in\Bar{B}^k$
\begin{equation}\label{6.26}
\begin{split}
-\frac{2}{c_0}Id&<D^2_sE_j^{t}(s,\{{X}_l(t)\}_{l=1}^k)\\
    &=D^2_sE_j^{t}(s,\{\delta_a(t)\Tilde{X}_l(t_a)+\delta_b(t)\Tilde{X}_l(t_b)\}_{l=1}^k)\\
    &<\delta^2_a(t)D^2_sE_j^{t}(s,\{\Tilde{X}_l(t_a)\}_{l=1}^k)+\delta^2_b(t)D^2_sE_j^{t}(s,\{\Tilde{X}_l(t_b)\}_{l=1}^k)\\
    &\leq -\frac{c_0}{2}Id.
\end{split}   
\end{equation}
The last inequality follows from $\delta^2_a(t)+\delta^2_b(t)\geq 1/2$
By ($\ref{6.26}$), we can choose $c=c_0/2$ such that $-\frac{1}{c}Id<D^2_sE_j^{t}(s,\{{X}_l(t)\}_{l=1}^k)<-cId,\:\forall s\in\Bar{B}^k.$ Now we define
\[\gamma_{j,s}(\cdot,t):=\Pi\circ\Big(\gamma_j(\cdot,t)+\sum_{l=1}^ks_lX_{l}(t)\Big),\quad s=(s_1,..,.s_k)\in\bar{B}^{k},\]
and let $E_j^t(s):=E(\gamma_{j,s}(\cdot,t)),$ then we have  
 \begin{enumerate}
    \item $E_j^t(s)$ has a unique maximum at $m_j(t)\in B^k_\frac{c}{\sqrt{10}}(0).$
    \item The map $\gamma_j(\cdot,t)\mapsto m_j(t)$ is continuous.
\item $\forall s\in \bar{B}^k$ and $\forall t\in U_{j,\delta/2}$ we have
		\begin{equation}\label{m'}
		-\frac{1}{c}Id\leq D^2E_j^t(s)\leq-c Id,
		\end{equation}
		and
		\begin{equation}\label{mm'}
		E_j^t(m_j(t))-\frac{1}{2c}|m_j(t)-s|^2\leq E^t_j(s)\leq E_j^t(m_j(t))-\frac{c}{2}|m_j(t)-s|^2.
		\end{equation}	
\end{enumerate}
Recall $\forall \{[k^l_i]\}_{i=0}^{m_l}\in K$ we have:
$d_B(\gamma_j(\cdot,t),\{k^l_i\}_{i=0}^{m_l})>\epsilon_l$ for $l=1,...,N_k,\forall t\in[0,1].$
		for all $j>j_k$. Without loss of generality (by rescaling $\{X_l(t)\}_{l=1}^k$ and $c$), we can assume that there exists $j_k'\in\mathbb{N}$ such that
		\begin{equation}\label{fork}
		d_B(\gamma_{j,s}(\cdot,t),\{k^l_i\}_{i=0}^{m_l})>\epsilon_l\quad\text{for }l=1,...,N_k,\forall s\in\bar{B}^k,\forall t\in[0,1],
		\end{equation}
		for all $j>j'_k$.
		\begin{claim}\label{MMarco}
			For $t$ belongs to the closure of $U_{j,\delta/2}\setminus U_{j,\delta/3}$, there exists a constant $\varepsilon=\varepsilon(\delta,\{u_i\}_{i=0}^n,\{X_l\}_{l=1}^k)>0$, so that for all $j\in\mathbb{N}$, $s\in\bar{B}^k$ satisfies
			\[E(\gamma_{j,s}(\cdot,t))\leq E(\gamma_j(\cdot,t))+\varepsilon,\]
			then 
			\[d_B(\gamma_{j,s}(\cdot,t),\{u_i\}_{i=0}^n)>2\varepsilon.\]
		\end{claim}
		\begin{proof}
			We argue by contradiction. Assume that there are sequences $s'_i\in\bar{B}^k$ and $t'_i\in I$ satisfying
			\begin{equation}\label{Marco}
			E(\gamma_{j,s'_i}(\cdot,t'_i))\leq E(\gamma_j(\cdot,t'_i))+1/i,
			\end{equation}
			and         
			\[d_B(\gamma_{j,s'_i}(\cdot,t'_i),\{u_i\}_{i=0}^n)\leq 2/i.\]
			Let $s_i'\to s=(s_1,...,s_k)\in\bar{B}^k$ and $t_i'\to t$. Hence we have $d_B(\gamma_{j,s}(\cdot,t),\{u_i\}_{i=0}^n)=0$,
			which implies that there are pairwise disjoint domains $\{\Omega_i\}_{i=0}^n$, $\cup_{i=0}^n\Omega_i\subset S^2$, and conformal dilations $\{D_i\}_{i=0}^n$ such that 
			$\sum_{i=0}^n\int_{S^2\setminus\Omega_i}|\nabla(u_i\circ D_i)|^2=0.$ For nontrivial harmonic sphere $u_i:S^2\to M$, $\int_{S^2\setminus\Omega_i}|\nabla(u_i\circ D_i)|^2=0$ is only possible when $\Omega_i=S^2$. That implies $\{u_i\}_{i=0}^n$ only contains one harmonic sphere, say $u_0$. The index assumption $\sum_{i=0}^n\text{Index}(u_i)=k$ now becomes $\text{Index}(u_0)=k$, and $\delta^2E(u_0)(X'_l,X_l')<0$ for $l=1,...,k$. Moreover, the $X_l(t)$ we constructed in (\ref{6.226}) is simply $X_l'$.
			By $\gamma_{j,s}(\cdot,t)\in C^0(S^2,M)\cap W^{1,2}(S^2,M)$, we have
			\begin{align*}
		   u_0(x)&= \gamma_{j,s}(x,t)\\
		   &=\Pi\circ\big(\gamma_j(x,t)+\sum_{l=1}^ks_lX_l'(x)\big),
			\end{align*}
 the equality holds for all $x\in S^2$. Moreover, because $\Pi$ is the nearest point projection defined on a tubular neighborhood of $M$ in $\mathbb{R}^N$, which implies the following
		\[\sum_{l=1}^k s_lX_l'(x)=u_0(x)-\gamma_j(x,t)+\nu(x),\quad\nu(x)\in T_{u_0(x)}^{\bot}M.\]
Let $Y(x):=\gamma_j(x,t)-u_0(x)$, then 
\[d\Pi_{u_0}(\sum_{l=1}^k s_l\nabla X_l')=d\Pi_{u_0}(-\nabla Y).\]
So we have
\begin{equation}\label{6.32}
\delta^2E(u_0)(Y,Y)<0,    
\end{equation}
if $s\neq 0$. (\ref{6.32}) implies that 
\begin{equation}
    E(\gamma_j(\cdot,t))=E(\Pi\circ(u_0+Y))\leq E(u_0),
\end{equation}
with the equality holds if and only if $s=0$. On the other hand, by assumption (\ref{Marco})
\[E(u_0)=E(\gamma_{j,s}(\cdot,t))\leq E(\gamma_j(\cdot,t)),\]
which forces $s$ to be $0$. Thus $\gamma_j(\cdot,t)=u_0$, which contradicts that $t$ belongs to the closure of $U_{j,\delta/2}\setminus U_{j,\delta/3}$.
\end{proof}
Now we consider the one-parameter flow
\begin{align*}
\{\phi_j^t(\cdot,x)\}_{x\geq 0}&\in\text{Diff}(\bar{B}^k)\\
 \phi_j^t(\cdot,\cdot):\bar{B}^k&\times [0,\infty)\to \bar{B}^k,   
\end{align*}
generated by the vector field:
\begin{equation}
s\mapsto -(1-|s|^2)\nabla E_j^t(s),\:s\in\bar{B}^k. 
\end{equation}
\begin{claim}\label{flow}		
For all $\kappa<\frac{1}{4}$, there is $T_j$ depending on $\big\{\{u_i\}_{i=0}^n,\{X_l\}_{l=1}^k,\kappa,\epsilon\big\}$ so that for any $t\in U_{j,\delta/2}$, and $v\in \bar{B}^k$ with $|v-m_j(t)|\geq\kappa$ we have:
\begin{equation}\label{decrease flow}
			E_j^t(\phi_j^t(v,T_j))<E_j^t(0)-\frac{c}{10}.
			\end{equation}
		\end{claim}
		\begin{proof}
By $m_j(t)\in B^k_\frac{c}{\sqrt{10}}(0)$ and  (\ref{mm'}) we know that for $\gamma_j(\cdot,t)$, $t\in U_{j,\delta/2}$, we have:
\begin{equation}\label{c}
\sup_{s\in\bar{B}^k}E_j^t(s)=E_j^t(m_j(t))\leq E_j^t(0)+\frac{c}{20}.
\end{equation}
So, to prove (\ref{decrease flow}), it suffices to show the existence of $T_j$ such that
\[|v-m_j(t)|\geq\kappa\implies E_j^t(\phi_j^t(v,T_j))<\sup_{s\in\bar{B}^k}E_j^t(s)-\frac{c}{5}.\]
We argue by contradiction and assume that there exists a constant  $\frac{1}{4}>\kappa>0$, a sequence $\{t_l\}_{l\in\mathbb{N}}\subset U_{j,\delta/2}$, and $\{s_l\}_{l\in\mathbb{N}}\subset\bar{B}^k$ with $|s_l-m_j(t_l)|\geq\kappa$ such that
			\begin{equation}\label{abc}
			E_j^{t_l}(\phi_j^{t_l}(s_l,l))\geq E_j^{t_l}(0)-\frac{c}{10}.
			\end{equation}
Combining (\ref{abc}) with (\ref{c}) we have $E_j^{t_l}(\phi_j^{t_l}(s_l,l))\geq E_j^{t_l}(m_j(t_l))-\frac{c}{5}$. Since $\phi_j^t(\cdot,\cdot)$ is an energy decreasing flow, we have
\[E_j^{t_l}(\phi_j^{t_l}(s_l,x))\geq E_j^{t_l}(\phi_j^{t_l}(s_l,l))\geq E_j^{t_l}(m_j(t_l))-\frac{c}{5},\:\forall 0\leq x\leq l.\]
Since both $U_{j,\delta/2}$ and $\bar{B}^k$ are compact, we obtain subsequential limits $t\in U_{j,\delta/2}$ and $s\in\bar{B}^k$ with
\begin{equation}\label{88}
E_j^t(\phi_j^t(s,x))\geq \sup_{|v|\leq 1}E_j^t(v)-\frac{c_0}{5},\quad\forall x\geq 0.
			\end{equation}
			Since $\gamma_j(\cdot,t)\mapsto m_j(t)$ is a continuous map, $|s_l-m_j(t_l)|\geq\kappa$ implies $|s-m_j(t)|\geq\kappa$ . Thus we have $\lim_{x\to\infty}|\phi_j^t(s,x)|=1$ and thus we deduce from the equation (\ref{88}) that
			\begin{equation}
			\sup_{|v|=1}E_j^t(v)\geq\sup_{|v|\leq 1}E_j^t(v)-\frac{c_0}{5}.
			\end{equation}
			On the other hand,  $m_j(t)\in B^k_\frac{c}{\sqrt{10}}(0)$ implies $|v-m_j(t)|>2/3$ for all $v\in\bar{B}^k$ with $|v|=1$. Hence, by equation \ref{mm'} we have
			\[\sup_{|v|=1}E_j^t(v)\leq\sup_{|v|\leq 1}E_j^t(v)-\frac{c_0}{2}\Big(\frac{2}{3}\Big)^2<\sup_{|v|\leq 1}E_j^t(v)-\frac{c_0}{5},\]
			which gives us the desired contradiction.
		\end{proof} 
		We define a continuous homotopy: 		
		$$H'_{j}:U_{j,\delta/2}\times[0,1]\longrightarrow B^k_{1/2^j}(0),$$
		so that
		$$H'_j(t,0)=0,\:\text{and}\:\inf_{t\in U_{j,\delta/2}}|H'_j(t,1)-m_j(t)|\geq\kappa_j>0.$$
		We are able to define $H_j'$ due to the assumption $\sum_i \text{Index}(u_i)=k\geq 2$. So we can choose a continuous path in $B^k_{1/2^j}(0)$ away from the curve of $m_j(t),\:t\in U_{j,\delta/2}.$ By claim \ref{flow}, there exists $T_j$ for $t\in U_{j,\delta/2}$ such that:
		$$E^t_j\big(\phi_j^t(H_j'(t,1),T_j)\big)<E_j^t(0)-\frac{c_0}{10}.$$
		Let $c_j:[0,1]\longrightarrow[0,1]$ be a cutoff function which is supported in $U_{j,\delta/2}$, and has value one in $U_{j,\delta/3}$, value zero in $[0,1]\setminus U_{j,\delta/2}.$
		Define:
		$$H_j(t,x)=H'_j(t,c_j(t)x),$$
		and
		\[H_j(t,x)=0\quad\forall t\in [0,1]\setminus U_{j,\delta/2}.\] 
		We now set  $s_j(t)=(s_j^1(t),...,s_j^k(t))\in\bar{B}^k$ to be \[s_j(t)=\phi_j^t(H_j(t,1),c_j(t)T_j),\quad\text{if }t\in U_{j,\delta/2},\]
		and 
		\[s_j(t)=0,\quad\text{if }t\in[0,1]\setminus U_{j,\delta/2}.\]
		We define $\gamma'_{j}(\cdot,t)$ to be:
		$$\gamma'_{j}(\cdot,t):=\Pi\circ\Big(\gamma_{j}(\cdot,t)+\sum_{l=1}^k s_j^l(t)X_l(t)\Big).$$
		Since $s_j$ is homotopic to the zero map in $\bar{B}^k$, so $\gamma'_{j}(\cdot,t)$ is homotopic to $\gamma_{j}(\cdot,t).$
		\begin{claim}
			$\{\gamma_j'(\cdot,t)\}_{j\in\mathbb{N}}$ is a minimizing sequence.
		\end{claim}
		\begin{proof}
			From the energy non-increasing property of $\{\phi_j^t(\cdot,x)\}\in\text{Diff}(\bar{B}^k)$ we have that for all $t\in[0,1]$
			\begin{equation}\label{mni}
			E(\gamma'_j(\cdot,t))=E_j^t(\phi_j^t(H_j(t,1),c_j(t)T_j))\leq E_j^t(H_j(t,1)).
			\end{equation}
			Moreover, from (\ref{2.229}), we know that there exists a continuous function $\Psi:[0,\infty)\to[0,\infty)$ with $\Psi(0)=0$ such that
			\[\Big|E(\gamma_{j,s}(\cdot,t))-E(\gamma_j(\cdot,t))\Big|\leq \Psi\Big(\|\sum_{l=1}^ks_lX_l\|_{W^{1,2}}\Big),\]
			and $H_j(t,1)\in B^k_{1/2^j}$ implies that
			\begin{equation}\label{mmmarco}
			\begin{split}
			 E(\gamma'_j(\cdot,t))&=E_j^t(\phi_j^t(H_j(t,1),c_j(t)T_j))\\
			 &\leq E_j^t(H_j(t,1))\\
&\leq E(\gamma_j(\cdot,t))+\Psi\Big(\frac{1}{2^{j}}\sum_{l=1}^k\| X_l\|_{W^{1,2}}\Big).   
			\end{split}
		\end{equation}
By (\ref{mmmarco}) and that $\gamma_j'(\cdot,t)$ is homotopic to $\gamma_j(\cdot,t)$ we have that
	\[W\leq\lim_{j\to\infty}\max_{t\in[0,1]}E(\gamma_j'(\cdot,t))=\lim_{j\to\infty}\max_{t\in[0,1]}E(\gamma_j(\cdot,t))=W,\]
			which finishes the proof
		\end{proof}
		\begin{claim}
			There exists $j_k'\in\mathbb{N}$ such that
			\[d_B(\gamma'_j(\cdot,t),\{k^l_i\}_{i=0}^{m_l})>\epsilon_l\quad\text{for }l=1,...,N_k,\forall t\in[0,1],\]
			for all $j>j_k$.
		\end{claim}
		\begin{proof}
			The claim follows from the assumption \ref{fork}.
		\end{proof}
		
		\begin{claim}
			there exists $\epsilon_J>0$ and $J\in\mathbb{N}$ such that
			\[d_B(\gamma_j'(\cdot,t),\{u_i\}_{i=0}^n)>\epsilon_J,\quad\forall t\in[0,1],\]
			for all $j>J$.
		\end{claim}
		\begin{proof}
			There are three cases to consider.

\textbf{case 1}
$t\in[0,1]\setminus U_{j,\delta/2}.$

$\gamma_j'(\cdot,t)=\gamma_j(\cdot,t)$ for all $j$, so there exists $\epsilon_1>0$ such that $d_B(\gamma'_j(\cdot,t),\{u_i\}_{i=0}^n)>\epsilon_1$ for all $j\in\mathbb{N}$.

\textbf{case 2}
$t\in U_{j,\delta/3}.$

By claim \ref{flow} we have
\begin{align*}
E(\gamma'_j(\cdot,t))&=E_j^t(\phi_j^t(H_j'(t,1),T_j))\\
&<E_j^t(0)-\frac{c}{10}\\
&=E(\gamma_j(\cdot,t))-\frac{c}{10},\quad\forall j\in\mathbb{N},    
\end{align*}
so 
\[\lim_{j\to\infty}\max_{t\in U_{j,\delta/3}}E(\gamma'_j(\cdot,t))<
\lim_{j\to\infty}\max_{t\in [0,1]}E(\gamma_j(\cdot,t))-\frac{c}{10}=W-\frac{c}{10}.\]
It implies that there exists  $\epsilon_2>0$ so $d_B(\gamma'_j(\cdot,t),\{u_i\}_{i=0}^n)>\epsilon_2$, or else by remark \ref{varifold energy} $\lim_{j\to\infty}\max_{t\in U_{j,\delta/3}}E(\gamma'_j(\cdot,t))=\sum_{i=0}^nE(u_i)=W$.
			
\textbf{case 3}
$t\in U_{j,\delta/2}\setminus U_{j,\delta/3}.$
			
By (\ref{mmmarco}) we have 
\[E(\gamma'_j(\cdot,t))\leq E(\gamma_j(\cdot,t))+\Psi\Big(\frac{1}{2^{j}}\sum_{l=1}^k\| X_l\|_{W^{1,2}}\Big),\]
for a continuous function $\Psi:[0,\infty)\to[0,\infty)$ with $\Psi(0)=0$.
and claim \ref{MMarco} implies that there exists $\epsilon_3>0$ so $d_B(\gamma'_j(\cdot,t),\{u_i\}_{i=0}^n)>\epsilon_3$ for $j$ sufficiently large.  

Let $\epsilon_J=\min\{\epsilon_1,\epsilon_2,\epsilon_3\}$, so we have 
\[d_B(\gamma_j(\cdot,t),\{u_i\}_{i=0}^n)>\epsilon_J,\:\forall t\in[0,1],\]
for $j$ sufficiently large.
\end{proof}
We have proved theorem \ref{deformation}.
	\end{proof}
\subsection{Proof of Theorem \ref{main}}
	\subsubsection*{Theorem \ref{main}}{\em Let $(M,g)$ be a closed Riemannian manifold of dimension at least three, $g$ generic and a nontrivial homotopy group $\pi_3(M)$, let $W$ be the width associated to the homotopy class $\Omega_\beta$ (see definition \ref{width}, (\ref{44.4})). Then there exists a collection of finitely many harmonic spheres $\{u_i\}_{i=0}^m,\: u_i:S^2\to M$, which satisfies the following properties:
		\begin{enumerate}
			\item $\sum_{i=0}^m E(u_i)=W,$
			\item $\sum_{i=0}^m Index(u_i)\leq 1.$
	\end{enumerate}}
	\begin{proof}
		Denote by $\mathcal{U}$ the collections of equivalent classes of harmonic spheres $\{[u_i]\}_{i=0}^n$ with $\sum_{i=0}^n\text{Index}(u_i)>1$ and $\sum_{i=0}^nE(u_i)=W$. By proposition \ref{countable}, $\mathcal{U}$ is countable and thus we can write  $\mathcal{U}=\Big\{\{[u^1_i]\}_{i=0}^{n_1},\{[u^2_i]\}_{i=0}^{n_2},...\Big\}$ with $\sum_{i=0}^{n_l}E(u^l_i)=W$ and $\sum_{i=0}^{n_l}\text{Index}(u_i^l)>1$ for each $l\in\mathbb{N}$. 
		
		Given a minimizing sequence $\{\gamma_j(\cdot,t)\}_{j\in\mathbb{N}}$, we consider  the collection of harmonic spheres: $\{[u_i^1]\}_{i=0}^{n_1}$, and by Theorem \ref{deformation} there exists $\{\gamma^1_j(\cdot,t)\}_{j\in\mathbb{N}}$ so that
		\begin{enumerate}
			\item $\gamma^1_j(\cdot,t)$ is homotopic to $\gamma_j(\cdot,t)$,
			\item $\{\gamma^1_j(\cdot,t)\}_{j\in\mathbb{N}}$ is a minimizing sequence,
			\item 
			there exists $\epsilon_1>0$ and $i_1\in\mathbb{N}$ such that
			\[d_B(\gamma_j^1(\cdot,t),\{u_i^1\}_{i=0}^{n_1})>\epsilon_1,\quad\forall t\in[0,1],\]
			for all $j>i_1$.
		\end{enumerate}
		We can apply Theorem \ref{deformation} again at the minimizing sequence 
		$\{\gamma_j^1(\cdot,t)\}_{j\in\mathbb{N}}$, with $\{[u^2_i]\}_{i=0}^{n_2}$ the given collection of harmonic spheres, and set the compact set of harmonic spheres $K$ to be $K^1:=\big\{\{[u_i^1]_{i=0}^{n_1}\}\big\}$, and obtain $\{\gamma_j^2(\cdot,t)\}_{j\in\mathbb{N}}$ so that
		\begin{enumerate}
			\item $\gamma^2_j(\cdot,t)$ is homotopic to $\gamma_j(\cdot,t)$,
			\item $\{\gamma^2_j(\cdot,t)\}_{j\in\mathbb{N}}$ is a minimizing sequence,
			\item 
			there exist $\epsilon_1,\epsilon_2>0$ and $i_1,i_2\in\mathbb{N}$ such that
			\[d_B(\gamma_j^1(\cdot,t),\{u_i^l\}_{i=0}^{n_l})>\epsilon_l,\quad j>i_l,\:\forall t\in[0,1],\]
			$l=1,2$.
		\end{enumerate}
		Proceeding inductively we can find $\{\gamma_j^m\}_{j\in\mathbb{N}}$ such that
		\begin{enumerate}
			\item $\gamma^m_j(\cdot,t)$ is homotopic to $\gamma_j(\cdot,t)$,
			\item $\{\gamma^m_j(\cdot,t)\}_{j\in\mathbb{N}}$ is a minimizing sequence,
			\item 
			there exist $\epsilon_l>0$ and $i_l\in\mathbb{N}$, $l=1,...,m$, such that
			\[d_B(\gamma_j^1(\cdot,t),\{u_i^l\}_{i=0}^{n_l})>\epsilon_l,\quad j>i_l,\:\forall t\in[0,1].\]
		\end{enumerate}
		We can choose an increasing sequence $p_m>i_m$ such that
		\[\max_{t\in[0,1]}E(\gamma^m_{p_m}(\cdot,t))\leq W+\frac{1}{m}.\]
		The sequence $\{\gamma^m_{p_m}(\cdot,t)\}_{m\in\mathbb{N}}$ is a minimizing sequence, thus by Theorem \ref{CDDD}, there exists a sequence $\{t_m\}_{m\in\mathbb{N}}\subset[0,1]$ and a collection of finitely many harmonic spheres $\{v_i\}_{i=0}^m$ with $\sum_{i=0}^mE(v_i)=W$, such that $\{\gamma^m_{p_m}(\cdot,t_m)\}_{m\in\mathbb{N}}$ bubble converges to up to subsequence. i.e.,
		\[d_B(\gamma^m_{p_m}(\cdot,t_m),\{v_i\}_{i=0}^m)\to 0,\:m\to\infty.\]
		Since $\gamma^m_{p_m}(\cdot,t)$ is away from $\mathcal{U}$ as $m\to\infty$ so we have
		$\sum_{i=0}^m Index(v_i)
		\leq 1$, this is what we wanted to prove.
	\end{proof}

\appendix
\section{}\label{F1.section}
	The goal of this section is to prove theorem \ref{CDDD}, that \textit{any minimizing sequence has a min-max sequence that bubble converges to a finite collection of harmonic spheres}. Theorem \ref{CDDD} doesn't follow immediately from Colding-Minicozzi's result \cite[theorem 1.8]{CD}, which states that there exists a sweepout, so that whenever the area of a slice of the sweepout is close to the width it must be close to a finite collection of harmonic spheres in bubble tree sense itself. \cite[theorem 1.8]{CD} is proven by showing that given any minimizing sequence $\{\gamma_j(\cdot,t)\}_{j\in\mathbb{N}}$, we can apply \textit{harmonic replacement}, so that the \textit{pulled-tight} sequence 
	$\tilde{\gamma}_j(\cdot,t)$ contains a min-max sequence, which is almost harmonic (see theorem \ref{almost harmonic}), thus bubble converges to a collection of finitely many harmonic spheres $\{u_i\}_{i=0}^n$ (whose sum of the energies realizes the width). That is,
	\[\{u_i\}_{i=0}^n\in\Lambda(\{\tilde{\gamma}_j(\cdot,t)\}_{j\in\mathbb{N}}).\]
Since $\tilde{\gamma}_j(\cdot,t)$ is obtained from $\gamma_j(\cdot,t)$ by doing harmonic replacement on the disjoint closed balls on $S^2$ with energy at most $\epsilon_1>0$ ($\epsilon_1$ as given in \cite[theorem 3.1]{CD}, so by \cite[theorem 3.1]{CD} we have
	\[E(\gamma_j(\cdot,t))-E(\tilde{\gamma}_j(\cdot,t))\geq\frac{1}{2}\int_{S^2}|\nabla\tilde{\gamma}_j(\cdot,t)-\nabla\gamma_j(\cdot,t)|^2.\]
		Combining with the assumption $\gamma_j$ is a minimizing sequence, we can conclude that
	\[\{u_i\}_{i=0}^n\in\Lambda(\{\gamma_j(\cdot,t)\}_{j\in\mathbb{N}}).\]
	We first list several technical results in \cite{CD} and then prove theorem \ref{CDDD}.
	\begin{theorem}[Compactness for almost harmonic maps, \cite{CD}]\label{almost harmonic}
	Suppose that $0<\epsilon\leq\epsilon_{su}$ ($\epsilon_{su}$ is a constant given by  \cite[theorem 3.2]{SU}), $E_0>0$ are constants and $\gamma_{j}:S^2\to M$ is a sequence of $C^0\cap W^{1,2}(S^2,M)$  maps with $E(\gamma_{j})\leq E_0$ satisfying:
		\begin{enumerate}
			\item $A(\gamma_{j})>E(\gamma_{j})-1/j$.
			\item For any finite collection of disjoint closed balls $B$ in $S^2$ with
			\[\int_{B}|\nabla\gamma_{j}|^2<\epsilon,\]
			there is an energy minimizing map $v:B\to M$ that equals to $\gamma_{j}$ on $\frac{1}{8}\partial B$ with
			\[\int_{\frac{1}{8}B}|\nabla\gamma_{j}-\nabla v|^2<\frac{1}{j}.\]
		\end{enumerate}
		Then a subsequence of $\gamma_{j}$ bubble converges to a finite collection of harmonic spheres $u_0,...,u_m:S^2\to M.$
	\end{theorem}

	\begin{theorem}[Colding-Minicozzi, theorem 3.1 \cite{CD}]\label{convex}
		There exists a contant $\epsilon_1>0$ (depending on $M$) so that if $u$ and $v$ are $W^{1,2}$ maps from $B_1\subset\mathbb{R}^2$ to $M$, $u$ and $v$ agree on $\partial B_1$, and $v$ is weakly harmonic with energy at most $\epsilon_1$, then
		\[\int_{B_1}|\nabla u|^2-\int_{B_1}|\nabla v|^2\geq\frac{1}{2}\int_{B_1}|\nabla v-\nabla u|^2.\]
	\end{theorem}

	\begin{theorem}[Colding-Minicozzi, \cite{CD}]\label{harmonicreplacem}
	There's a constant $\epsilon_0>0$ and a continuous function $\Phi:[0,\infty)\to[0,\infty)$ with $\Phi(0)=0,$ both depending on $M$, so that given any $\gamma\in\Omega_\beta$ without nonconstant harmonic slices and $\beta\in\pi_3(M)$ nontrivial, there exists $\tilde{\gamma}\in\Omega_\beta$ so that
	\[E(\tilde{\gamma}(\cdot,t))\leq E(\gamma(\cdot,t)),\]
	for each $t$ and so for each $t$ with $E(\tilde{\gamma}(\cdot,t))\geq W/2$ we have the following:
If $\mathcal{B}$ is any finite collection of disjoint  closed balls in $S^2$ with $\int_{\mathcal{B}}|\nabla\gamma(\cdot,t)|^2<\epsilon_0$ and $v:\cup_{B\in\mathcal{B}}\frac{1}{8}B\to M$ is an energy minimizing map equal to $\gamma(\cdot,t)$ on $\cup_{B\in\mathcal{B}}\frac{1}{8}\partial B$, then
	\[\int_{\cup_{B\in\mathcal{B}}\frac{1}{8}B}|\nabla\tilde{\gamma}(\cdot,t)-\nabla v|^2\leq\Phi(E(\tilde{\gamma}(\cdot,t))-E(\gamma(\cdot,t))).\]
\end{theorem}

\begin{proposition}[Proposition 1.2 \cite{CD}]\label{A=E}
Given a closed manifold $M$ with dimension $n\geq 3$, and a map $\beta\in\Omega$ representing a nontrivial class in $\pi_3(M)$. The width of energy $W$, and the width of area $W_A$ associated to the homotopy class $\Omega_\beta$ are equal.    
	\end{proposition}
	\begin{remark}
		Let $\{\gamma_j\}_{j\in\mathbb{N}}$ be a minimizing sequence for $W$. Since $\text{Area}(\gamma_j(\cdot,t))\leq E(\gamma_j(\cdot,t))$, and $W_A=W$ by Proposition \ref{A=E}, we have that
		\[W_A\leq\lim_{j\to\infty}\max_{s\in[0,1]}\text{Area}(\gamma_j(\cdot,s))\leq\lim_{j\to\infty}\max_{s\in[0,1]}E(\gamma_j(\cdot,s))=W_A,\]
		which implies that $\{\gamma_j\}_{j\in\mathbb{N}}$ is also a minimizing sequence for $W_A$.
	\end{remark}

	\begin{theorem}[Colding-Minicozzi, theorem 1.8 \cite{CD}]\label{CD}
		Given a closed manifold $M$ with dimension $n\geq 3$, and a map $\beta\in\Omega$ representing a nontrivial class in $\pi_3(M)$, there exists a sequence of sweepouts $\gamma_{j}(\cdot,t)\in\Omega_\beta$ with $\max_{s\in[0,1]}\gamma_j(\cdot,s)\to W$, so that given $\epsilon>0$, there exists $\bar{j}$ and $\delta>0$ so that if $j>\bar{j}$ and
		\begin{equation*}
		Area(\gamma_j(\cdot,s))>W-\delta,
		\end{equation*}
		then there are finitely many harmonic spheres $u_i:S^2\to M$, $i=0,...,n$, with 
		\begin{equation*}
		d_V(\gamma_{j}(\cdot,s),\{u_i\}_{i=0}^n)<\epsilon.
		\end{equation*}
	\end{theorem}

\begin{corollary}\label{pulltight}
		Given a sweepout $\{\gamma_j(\cdot,t)\}_{j\in\mathbb{N}}$, the pulled-tight of it: $\{\tilde{\gamma}_j(\cdot,t)\}_{j\in\mathbb{N}}$ given by \cite[Theorem 2.1]{CD} has the following property:
	\[E(\gamma_j(\cdot,t))-E(\tilde{\gamma}_j(\cdot,t))\geq\frac{1}{2}\int_{S^2}|\nabla\tilde{\gamma}_j(\cdot,t)-\nabla\gamma_j(\cdot,t)|^2.\]
	\end{corollary}
	\begin{proof}
		The pulled-tight sweepout $\tilde{\gamma}_j(\cdot,t)$ is constructed by doing harmonic replacement on $\gamma_j(\cdot,t)$ over domain $\mathcal{B}\subset S^2$ with the energy of $\gamma_j(\cdot,t)$ on $\mathcal{B}$ at most $\epsilon_1>0$. Then by theorem \ref{convex} the corollary follows.
	\end{proof}
	With the above observation, we now state and prove theorem \ref{CDDD}:
	\begin{theorem}
		Given a closed manifold $M$ with dimension at least three, and a map $\beta\in\Omega$ representing a nontrivial class in $\pi_3(M)$, then for any sequence of sweepouts $\gamma_j\in\Omega_\beta$ with 
		\[\lim_{j\to\infty}\max_{s\in[0,1]}E(\gamma_j(\cdot,s))= W,\] 
		there exists a subsequence $\{i_j\}\to\infty$,  $t_{i_j}\in[0,1]$, and a collection of finitely many harmonic spheres $\{u_i\}_{i=0}^n$ such that
		\[d_B(\gamma_{i_j}(\cdot,t_{i_j}),\{u_i\}_{i=0}^n)<1/j,\]
		and
		\[\sum_{i=0}^nE(u_i)=W.\]
	\end{theorem}
	\begin{proof}
		Given a minimizing sequence $\{\gamma_j(\cdot,t)\}_{j\in\mathbb{N}}$ such that
		\begin{equation}\label{pt}
		\max_{s\in[0,1]}E(\gamma_j(\cdot,s))<W+1/2j.
		\end{equation}
		Applying theorem \ref{harmonicreplacem} to $\gamma_j$ gives a sequence $\tilde{\gamma}_j\in\Omega_\beta$ with 
		\[E(\tilde{\gamma}_j(\cdot,t))\leq E(\gamma_j(\cdot,t)).\]
	We choose $t_j\in[0,1]$ so that $\text{Area}(\tilde{\gamma}_j(\cdot,t_j))=\max_{s\in[0,1]}\text{Area}(\tilde{\gamma}_j(\cdot,s))$. By proposition \ref{A=E} we have:
\begin{equation}\label{68}
\begin{split}
    W\leq\text{Area}(\tilde{\gamma}_{j}(\cdot,t_j))
    &\leq E(\tilde{\gamma}_{j}(\cdot,t_j))\\
    &\leq E(\gamma_{j}(\cdot,t_j))\\
    &\leq \max_{s\in[0,1]}E(\gamma_j(\cdot,s))<W+1/2j,
\end{split}    
\end{equation}
thus imply that $\text{Area}(\tilde{\gamma}_{j}(\cdot,t_j))>E(\tilde{\gamma}_{j}(\cdot,t_j))-1/j$ and if $\mathcal{B}$ is any finite collection of disjoint  closed balls in $S^2$ with $\int_{\mathcal{B}}|\nabla\gamma(\cdot,t)|^2<\epsilon_0$ and $v:\cup_{B\in\mathcal{B}}\frac{1}{8}B\to M$ is an energy minimizing map equal to $\gamma(\cdot,t)$ on $\cup_{B\in\mathcal{B}}\frac{1}{8}\partial B$, then
\begin{align*}
\int_{\cup_{B\in\mathcal{B}}\frac{1}{8}B}|\nabla\tilde{\gamma}(\cdot,t)-\nabla v|^2\leq&\Phi(E(\tilde{\gamma}(\cdot,t))-E(\gamma(\cdot,t)))\\
=&\Phi(1/2j).
\end{align*}
(see theorem \ref{harmonicreplacem}). So $\{\gamma_j(\cdot,t_j)\}_{j\in\mathbb{N}}$ is an almost harmonic sequence, and by theorem \ref{almost harmonic}, there exists a collection of finitely many harmonic spheres $\{u_i\}_{i=0}^n$ and subsequence $i_j\to\infty$ so that 
	\[d_B(\gamma_{i_j}(\cdot,t_{i_j}),\{u_i\}_{i=0}^n)<1/2j.\]
		By corollary \ref{pulltight} and (\ref{68}) we have that 
	\[\frac{1}{2}\int_{S^2}|\nabla\tilde{\gamma}_j(\cdot,t_j)-\nabla\gamma_j(\cdot,t_j)|^2\leq E(\tilde{\gamma}_j(\cdot,t_j))-E(\gamma_j(\cdot,t_j))<1/j.\]
		Thus we have 
		\[d_B(\tilde{\gamma}_{i_j}(\cdot,t_{i_j}),\{u_i\}_{i=0}^n)\to 0,\quad j\to\infty,\]
		which is the desired result.
	\end{proof}

\section{}\label{F2}
We are going to prove in this section that for a closed manifold of dimension at least three, with a generic metric, then the set of all harmonic spheres up to equivalent class with bounded energy is countable (proposition \ref{countable}). This result is expected, given that we know a similar result holds for minimal embedded hypersurfaces. Namely:
\begin{theorem}[\cite{BS}]\label{BenSharp}
Given a closed manifold $(M^n,g)$, $3\leq n\leq 7$, with $g$ generic, the set of embedded minimal hypersurfaces with bounded area and index is finite.
	\end{theorem}
However, harmonic spheres with bounded energy are merely branched minimal immersions with bounded area. Proposition \ref{countable} doesn't follow from theorem \ref{BenSharp}. The proof of Proposition \ref{countable} brings together several important results like bumpy metric theorem for minimal surface \cite[theorem 2.1]{BW}, bumpy metric theorem for harmonic map \cite{JD} and comparison between second variation of area and energy of a minimal surface \cite{Mario}. We first recall some basic properties for harmonic maps like tension field and Jacobi field, state bumpy metric theorem for harmonic maps and minimal submanifold, and then prove proposition \ref{countable}.

	\begin{theorem}[A Priori Estimate \cite{SU} Main Estimate 3.2]\label{regularity}
		Given $\Sigma$ and $M$, there exist $\epsilon_{su}>0$ and $\rho>0$ such that if $r_0<\rho$, $u:\Sigma\to M$ is harmonic and
		\[\int_{B_{r_0}^{\Sigma}(y)}|\nabla u|^2<\delta\epsilon_{su},\]
		then
		\[|\nabla u|^2(y)\leq\frac{\delta}{r_0^2}.\]
	\end{theorem}
\begin{remark}
Once we know that $\nabla u \in L^{\infty}_{\text{loc}}$, it then follows from equation (\ref{whel}) that $\Delta u\in L^{\infty}_{\text{loc}}$, which implies by standard estimates on the inverse of the Laplacian that $u\in W^{2,p}_{\text{loc}}$ for all $p<\infty$. Hence we deduce that $\Delta u\in W^{1,p}_{\text{loc}}$ and hence $u\in W^{3,p}_{\text{loc}}$ for all $p>0$. We can then repeat this argument to show that $u\in W^{r,p}_{\text{loc}}$, $\forall r$, and so the smoothness of the solution follows. 
\end{remark}
	
\begin{definition}[tension field] We call the tension field of $f$ to be the following
\[\tau(f):=\text{trace}\nabla df.\]
\end{definition}
\begin{remark}
Intrinsically, the Euler-Lagrange equation for energy is
\begin{equation}
    \tau(f)=0.
\end{equation}
\end{remark}
\begin{definition}[Jocabi operator]
For a harmonic map $f:M\to N$, we define the Jacobi operator of $f$ to be
\begin{equation}\label{jocabi}
\mathcal{J}_f(V):=-\Delta V-\text{trace}_MR^N(V,df)df,
\end{equation}
here $\Delta$ is the Laplacian on sections of $f^{-1}TN$ given in local coordinates on $M$ by 
\[\Delta=h^{\alpha\beta}(f^*\nabla^N)_{\frac{\partial}{\partial x^\alpha}}(f^*\nabla^N)_{\frac{\partial}{\partial x^\beta}},\]
so we have 
\[I(V,W)=\int_M\langle\mathcal{J}_f(V),W\rangle dM.\]
\end{definition}
Let
\begin{equation}\label{2.16}
    \begin{split}
        &f_{st}(x)=f(x,s,t),\\
        &f:M\times(-\epsilon,\epsilon)\times(-\epsilon,\epsilon)\to N
    \end{split}
\end{equation}
be a smooth family of maps between Riemannian manifolds of finite energy. $M$ may have nonempty boundary, in which case we require $f(x,s,t)=f(x,0,0)$ for all $x\in\partial M$ and all $s,t$.
\begin{proposition}\label{prop7}
For a smooth family of maps $f_{st}:M\to N$ defined as \ref{2.16}, with
\begin{align*}
    V:=&\frac{\partial}{\partial s}\at[\Big]{s=t=0}f_{st}\\
    W:=&\frac{\partial}{\partial t}\at[\Big]{s=t=0}f_{st}.
\end{align*}
Let $f_{00}=f$ be a smooth harmonic map. We have the Jacobi operator $\mathcal{J}_f$ defined as (\ref{jocabi}) to be the following:
\[\mathcal{J}_f(V)=-\frac{\partial}{\partial s}\at[\Big]{s=0}\tau(f_{s0}).\]
\end{proposition}
\begin{proof}
The proposition follows from the computation below:
\begin{align*}
\frac{\partial^2}{\partial s\partial t}\at[\Big]{s=t=0}E(f_{st})&=\frac{\partial}{\partial s}\at[\Big]{s=0}\int_M\langle df_{st},\nabla_{\frac{\partial}{\partial t}}df_{st}\rangle\at[]{t=0}\\
    &=-\frac{\partial}{\partial s}\at[\Big]{s=0}\int_M\langle \tau(f_{s0}),\frac{\partial}{\partial t}\at[\Big]{t=0}f_{st}\rangle\\
    &=\int_M\langle-\frac{\partial}{\partial s}\at[\Big]{s=0}\tau(f_{s0}),\frac{\partial}{\partial t}\at[\Big]{s=t=0}f_{st}\rangle dM+\int_M\langle\tau(f_{00}),-\frac{\partial^2}{\partial s\partial t}\at[\Big]{s=t=0}f_{st}\rangle dM\\
    &=\int_M\langle-\frac{\partial}{\partial s}\at[\Big]{s=0}\tau(f_{s0}),W\rangle dM\\
    &=\int_M\langle\mathcal{J}_f(V),W\rangle dM.
\end{align*}
The second to last equality follows by $f_{00}$ is harmonic. Then it implies that 
\[\mathcal{J}_f(V)=-\frac{\partial}{\partial s}\at[\Big]{s=0}\tau(f_{s0}).\]
\end{proof}
\subsection{Bumpy Metric Theorem for Minimal Submanifolds}
Now we define generic metric for a specific metric space. For $k\in\mathbb{N}$, the space of $L^2_k$ Riemannian metrics on $M$ simply denotes an open set of the Hilbert space of $L^2_k$-sections of the second symmetric power of $T^*M$.
\begin{definition}[Generic Metrics, \cite{JD}]
By a generic Riemannian metric on a smooth manifold $M$ we mean a Riemannian metric that belongs to a countable intersection of open dense subsets of the spaces of $L^2_k$ Riemannian metrics on $M$ with the $L^2_k$ topology, for some choice of $k\in\mathbb{N}$, $k\geq 2$.
\end{definition}
\begin{remark}
Notice that \textit{generic metric} always implies a countable intersection of open dense subsets of the metric space. For geodesics and harmonic maps the metric space is $L^2_k$ \cite{JD}, and for minimal submanifolds it's $C^{q}$ Riemannian metric for $q\geq 3$, \cite{BW}. 
\end{remark}

\begin{definition}[Bumpy Metric \cite{BW}]\label{bwbumpy}
 A metric $g$ on $M$ is called \textit{bumpy} if there is no smooth immersed minimal submanifold (minimal with respect to $g$) with a non-trivial Jacobi field.  
\end{definition}
\begin{theorem}[Bumpy Metric Theorem for Minimal Submanifold, \cite{BW}]\label{bmt}
If $M$ is a compact manifold, then for a generic choice of metric of $C^q$ $g$ on $M$ ($q\geq 3$), there are no minimal submanifolds with nonzero normal Jacobi fields. That is, each minimal submanifold has nullity $0$.
\end{theorem}
\subsection{Bumpy Metric Theorem for Harmonic Maps}
To state the bumpy metric theorem for harmonic maps, first we need to define prime harmonic map. 
\begin{definition}[Chapter 5 \cite{JD}]
			Suppose that $f,h:S^2\to M$ are harmonic spheres. $f$ is called a \textit{branched cover} of $h$ if there exists a holomorphic map $g:S^2\to S^2$ of degree $d\geq 2$ such that $f=h\circ g$. We call $f$ a \textit{prime} harmonic sphere if it's not a branched cover of another harmonic sphere.
		\end{definition}
		\begin{definition}[branch point] A point $p\in\Sigma$ is a \textit{branch point} for the harmonic map $f:\Sigma\to M$ if $(\partial f/\partial z)(p)=0$, where $z$ is any complex coordinate near $p$.
		\end{definition}
		
		\begin{definition}
			If $p$ is a branch point of $f$ but there exists some neighborhood $V$ containing $p$ such that $f(V)$ is an immersed surface, then we say that $p$ is a \textit{false branch point}.
		\end{definition}
		\begin{definition}[Chapter 5 \cite{JD}]
			If $f:\Sigma\to M$ is a parametrized minimal surface we say that $p\in\Sigma$ is an \textit{injective point} for $f$ if 
			\[df(p)\neq 0,\text{ and}\quad f^{-1}(f(p))=0.\]
			If $f:\Sigma\to M$ is connected and has injective points, we say it is \textit{somewhere injective}.
		\end{definition}
		\begin{lemma}[Chapter 5 \cite{JD}]\label{2}
			If a harmonic map $f:\Sigma\to M$ is somewhere injective, then it is prime.
		\end{lemma}
		\begin{lemma}[Chapter 5 \cite{JD}]\label{bp}
			If a conformal harmonic map $f:\Sigma\to M$ is prime, its injective points form an open dense subset of $\Sigma$.
		\end{lemma}
		\begin{theorem}[Bumpy Metric Theorem for Harmonic Maps, theorem 5.1.1\cite{JD}]\label{jdbumpy}
			If $M$ is a compact smooth manifold of dimension at least three, then for a generic choice of Riemannina metric on $M$, all prime compact parametrized minimal surfaces $f:\Sigma\to M$ are free of branch points and lie on nondegenerate critical submanifolds, each such submanifold being an orbit of the group $G$ of conformal automorphisms of $\Sigma$ which are homotopic to the identity. 
		\end{theorem}
		\begin{remark}
			In Theorem \ref{jdbumpy}, nondegeneracy of a prime harmonic map means that the Jacobi field of it are those generated by the automorphisms of $\Sigma.$
		\end{remark}
		In our case $\Sigma=S^2$ and $G=PSL(2,\mathbb{C}).$ Given a harmonic sphere $u:S^2\to M$, if $u$ is prime then it's free of branch points by Theorem \ref{jdbumpy}. If $u$ is not prime it can be written as a branched cover of a prime harmonic sphere, and all of its branched points are false. Theorem \ref{jdbumpy} implies that for all harmonic spheres $u:S^2\to M$, $u(S^2)$ is an smooth immersed minimal submanifold. 
		
		We define the nullity
		$\mathcal{N}$ of a functional at a critical point $u$ is the dimension of the space of Jacobi fields of the functional at $u$. 
		\begin{theorem}[theorem 3.1\cite{Mario}]\label{EEquality}
			Given a harmonic sphere $u:S^2\to M$, let
			\begin{align*}
			\mathcal{N}_A=&\text{nullity of }u\text{ as a critical point of the area functional }A,\\
			\mathcal{N}_E=&\text{nullity of }u\text{ as a critical point of the energy functional }E,\\
			\mathcal{N}_E^T=&\text{dimension of the space of purely tangential
				Jacobi fields of }u,\\
			&\textit{
				as a critical point of }E.
			\end{align*}
			Then
			\[\mathcal{N}_A=\mathcal{N}_E-\mathcal{N}_E^T.\]
		\end{theorem}
		\subsection{Countability of Harmonic Spheres with Bounded Energy}
		
		\begin{lemma}\label{bumpyy}
			Given $(M,g)$, $g$ bumpy, the dimension of $M$ is at least three, and a harmonic sphere $u:S^2\to M$. There exists $\epsilon=\epsilon(u)>0$, so that if $\|f-u\|_{W^{1,2}}<\epsilon$, $f:S^2\to M$ harmonic, then $[f]=[u].$ (See definition \ref{equivalentclass} for $[u]$.)
		\end{lemma}
		\begin{proof}
			Given a harmonic sphere $u:S^2\to M$, Theorem \ref{jdbumpy} implies that $u(S^2)$ is a smooth immersed minimal sphere.
			We argue by contradiction. If not, then there exists a sequence of harmonic spheres $\{f_i\}_{i\in\mathbb{N}}$, $f_i:S^2\to M$, such that $\|f_i-u\|_{W^{1,2}}<1/i$ and $[f_i]\neq[u]$ for all $i\in\mathbb{N}.$ By strong convergence in $W^{1,2}$, Theorem \ref{regularity}, and Arzel\`a-Ascoli theorem we know that the convergence $f_i\to u$ is smooth and uniform. Thus we can choose subsequence $i(j)\to\infty$ as $j\to\infty$ such that 
			\begin{equation}\label{souri}
			|f_{i(j)}(x)-u(x)|<1/j,\:\forall x\in S^2,    
			\end{equation}
			and
			\begin{equation}
			\Big|\frac{\partial f_{i(j)}(x)}{\partial x^{\alpha}}-\frac{\partial u(x)}{\partial x^{\alpha}}\Big|<1/j,\:\forall x\in S^2,
			\end{equation}
			where $(x^\alpha)$ denotes the local coordinate system of $S^2$. Since $f_{i(j)}$ is harmonic we have that 
			\begin{align*}
			0-0&=\tau(f_{i(j)})-\tau(u)\\
			&=\int_{0}^1\frac{\partial}{\partial s}\at[\Big]{s=t}\tau(u+s(f_{i(j)}-u))dt.
			\end{align*}
			Thus there exists $t_j\in(0,1)$ such that 
			\[0=\frac{\partial}{\partial t}\at[\Big]{t=t_j}\tau(u+t(f_{i(j)}-u)).\]
Let 
			\[w:=\lim_{j\to\infty}\frac{f_{i(j)}-u}{\max_{x\in S^2}|f_{i(j)}(x)-u(x)|}.\]
By proposition \ref{prop7}, we know that $w$ is an nontrivial Jacobi field. i.e., $\mathcal{J}_u(w)=0.$ Assume that $f_{i(j)}(S^2)\neq u(S^2)$ for all $j\in\mathbb{N}$. Since $f_i(S^2)\neq u(S^2)$ implies that $w$ is not purely tangential, Theorem \ref{EEquality} implies that $u$ has a nontrivial Jacobi field for area functional as a smooth immersed minimal surface. This contradicts the bumpy metric assumption.
			
			Now we consider the case that $f_{i(j)}(S^2)=u(S^2)$ for all $j\in\mathbb{N}$. If $u$ is a prime harmonic sphere, then Theorem \ref{jdbumpy} implies that $w$ is generated by $PSL(2,\mathbb{C})$, contradicting the assumption $[f_{i(j)}]\neq[u]$. Now we consider the case that $u$ is not prime. \begin{claim}
				For $j$ sufficiently large, we have
				\[deg(f_{i(j)})=deg(u).\]
			\end{claim}
			\begin{proof}[Proof of the claim]
				It's known that if two maps are homotopic then they have the same degree. The stradegy of the proof is similar and can be found in \cite[Chapter 5]{dt}.
				
				For $j$ sufficiently large, by equation \ref{souri} we know that $f_{i(j)}$ is homotopic to $u$. That is, there exists a continuous map $H_j:S^2\times[0,1]\to u(S^2)$:
				\[H_j(x,t):=\gamma_x^j(t),\:x\in S^2,\]
				such that $H_j(x,0)=u(x)$ and $H_j(x,1)=f_{i(j)}(x)$ for all $x\in S^2$, where $\gamma_x^j(t)$ denotes the unique geodesic with respect to the intrinsic metric starting from $u(x)$ with end point at $f_{i(j)}(x)$. Since $f_{i(j)}(S^2)=u(S^2)$, and by Lemma \ref{bp} we can choose $y\in u(S^2)$ so it satisfies the following conditions:
				\begin{enumerate}
					\item $\forall p\in u^{-1}(y)$, $du(p)\neq 0$,
					\item $\forall p\in f_{i(j)}^{-1}(y)$, $df_i(p)\neq 0$,
					\item $y$ is a regular value for $H_j$ and $H_j|_{\partial(S^2\times[0,1])}$.
				\end{enumerate}
				$H_j^{-1}(y)$ is a compact 1-dimensional submanifold with boundary $(H_j|_{\partial(S^2\times[0,1])})^{-1}(y)$. In other words, it contains embedded arcs which are transverse to $\partial(S^2\times[0,1])$. By \cite[Chapter 5]{dt}, given $p_1\in H_j^{-1}(y)$, there is an unique $p_2\in H_j^{-1}(y)$, $p_1\neq p_2$, and a component arc $\Gamma\in H_j^{-1}(y)$ with $\partial\Gamma=\{p_1,p_2\}$. By \cite[Chapter 5, Lemma 1.2]{dt}, $p_1$ and $p_2$ are of opposite type for $H_j|_{\partial(S^2\times[0,1])}$. This implies that for any $p_1\in u^{-1}(y)$, there is an unique corresponding $p_2\in f_{i(j)}^{-1}(y)$. Then it follows that degree of $u$ and degree of $f_{i(j)}$ are the same for $j$ sufficiently large.  \end{proof}
			By the smooth convergence of $f_{i(j)}\to u$, and $deg (f_{i(j)})=deg(u)$. Let $u=\phi\circ\tilde{u}$ for some prime harmonic map $\tilde{u}$. We can write $f_{i_{j}}$ as a branched cover of a prime harmonic map $h_{i_{j}}$,i.e., $f_{i_{j}}=g_{i_j}\circ h_{i_{j}}$, with $deg (f_{i(j)})=deg(g_{i_j})=deg(u)$. By lemma \ref{2} and lemma \ref{bp} we know that $h_{i_{j}}$ is somewhere injective and the injective points of $h_{i_{j}}$ form an open dense subset of $S^2$. So we obtain a sequence of prime harmonic spheres $h_{i_{j}}$ that converges strongly to a prime harmonic sphere $\tilde{u}$ in $W^{1,2}$. It is the desired contradiction.
		\end{proof}	
		
		\begin{corollary}\label{same}
The $\epsilon(u)>0$ given as lemma \ref{bumpyy} is invariant under the equivalence relation. i.e., given harmonic spheres $f,g:S^2\to M$, if $[f]=[g]$, then $\epsilon(f)=\epsilon(g).$ 
		\end{corollary}
		\begin{proof}
We argue by contradiction. Given harmonic spheres $f,g:S^2\to M$ with $[f]=[g]$, and assume $\epsilon(f)>\epsilon(g)$. Then there exists a harmonic map $u:S^2\to M$ with
\[\epsilon(g)<\int_{S^2}|\nabla u-\nabla g|^2<\epsilon(f),\]
such that $[u]\neq [g].$ Since $[g]=[f]$, there exists $\phi\in PSL(2,\mathbb{C})$ such that $f\circ\phi=g$. This implies 
\begin{align*}
\epsilon(g)<\int_{S^2}|\nabla u-\nabla g|^2&=\int_{S^2}|\nabla u-\nabla(f\circ\phi)|^2\\
&=\int_{S^2}|\nabla(u\circ\phi^{-1})-\nabla f|^2<\epsilon(f),    
\end{align*}
thus we have $[u]=[u\circ\phi^{-1}]=[f]=[g]$, which is the desired contradiction.

		\end{proof}
		
		\begin{definition}
			We define $\mathcal{F}^W$ to be the equivalent classes of all harmonic sphere with energy bound $W$, that is:			$$\mathcal{F}^W:=\Big\{[f]\mid f:S^2\rightarrow M\text{ harmonic }, E(f)\leq W\Big\}.$$
		\end{definition}
		\begin{proposition}\label{countable}
			Given a closed manifold $(M,g)$, with generic $g$ and the dimension of $M$ is at least three. The set $\mathcal{F}^W$ is countable.
	\end{proposition}
		\begin{proof}
			We pick a finite set $\{x^1_n,...,x^{p^n}_n\}$ of $S^2$ so that $S^2\subset\bigcup_{k=1}^{p^n} B_{1/n}(x^k_n).$ ($B_{1/n}(x^k_n)$ is the geodesic ball centered at the point $x^k_n$.) We define $\mathcal{F}^W(n)$ to be:
			$$\mathcal{F}^W(n)=\Big\{[f]\mid f\in\mathcal{F}^W,\int_{B_{1/n}(x^k_n)}|\nabla f|^2<\epsilon_{su},\:\text{for }k=1,...,p^n\Big\},$$
			where $\epsilon_{su}>0$ is the constant given in Theorem \ref{regularity}. We can see that:$$\mathcal{F}^W=\bigcup_{n\in\mathbb{N}}\mathcal{F}^W(n).$$ 
			Now we prove that $\mathcal{F}^W(n)$ is finite. If not, then there exists a sequence $\{[f_i]\}_{i\in\mathbb{N}}\subset \mathcal{F}^W(n)$, and $[f_i]\neq[f_j]$ if $i\neq j$. By Theorem \ref{bubble conv}, $f_i$ bubble converges to a harmonic map $f\in\mathcal{F}^W$ up to subsequence. Because of the assumption:
			$$\int_{B_{1/n}(x^k_n)}|\nabla f_i|^2<\epsilon_{su}\quad\forall i\in\mathbb{N},k=1,...,p^n.$$ 
			By Theorem \ref{regularity}, this implies that the convergence is strong in $W^{1,2}$ and $f\in\mathcal{F}^W(n)$. Then it contradicts Lemma \ref{bumpyy}.
		\end{proof}	

		\bibliographystyle{alpha}
		\bibliography{main.bib}	
\end{document}